\documentclass[11 pt,english]{article}

\usepackage[margin=1in]{geometry}                % See geometry.pdf to learn the layout options. There are lots.
\usepackage{subfigure}
\usepackage{tikz}

\usetikzlibrary{calc, positioning, through, math, arrows, backgrounds}
\tikzstyle{vtx}=[inner sep=1pt,draw, shape=circle, font=\tiny]
\tikzstyle{line}=[inner sep=3pt,draw, shape=rectangle, line width = 3pt]
\tikzset{>=stealth}
\tikzstyle{lbl}=[inner sep = 1 pt, fill = white, midway]

\usepackage[activeacute]{babel}
\usepackage{amsfonts} 
\usepackage{mathrsfs} 
\usepackage{amssymb, amsmath, amsfonts}
\usepackage{makeidx}
\usepackage{epsfig}
\usepackage{color}
\usepackage[T1]{fontenc}
\usepackage{graphics,graphicx,graphpap}
\usepackage[ansinew]{inputenc}
\usepackage{amsthm}
\usepackage{hyperref}

\newcommand\blfootnote[1]{%
  \begingroup
  \renewcommand\thefootnote{}\footnote{#1}%
  \addtocounter{footnote}{-1}%
  \endgroup
}

\if@twoside \oddsidemargin 10pt \evensidemargin 10pt \marginparwidth 20pt
 \else \oddsidemargin 10pt \evensidemargin 10pt \marginparwidth 20pt \fi

\newcommand{\ZZ}{\mathbb{Z}}

\newcommand{\Aut}{\mathrm{Aut}}
\newcommand{\G}{\Gamma}

\newcommand{\cB}{\mathcal{B}}
\newcommand{\cR}{\mathcal{R}}
\newcommand{\cC}{\mathcal{C}}

\newcommand{\la}{\langle}
\newcommand{\ra}{\rangle}

\renewcommand{\t}{\tilde}

\newcommand{\WH}{\mathop{\rm WH}}
\newcommand{\lcm}{\mathop{\rm lcm}}

\newtheorem{theorem}{Theorem}[section]
\newtheorem{proposition}[theorem]{Proposition}

\newtheorem{lemma}[theorem]{Lemma}
\newtheorem{problem}[theorem]{Problem}

\newtheorem{observation}[theorem]{Observation}

\theoremstyle{definition}

\newtheorem{construction}[theorem]{Construction}

\begin{document}

\begin{center}
\Large{\textbf{Symmetries of the Woolly Hat graphs}} \\ [+4ex]
\end{center}

\begin{center}
Leah Wrenn Berman{\small $^{a}$},\   
Hiroki Koike{\small $^{b}$},\
El\'ias Moch\'an{\small $^{c}$},\ 
Alejandra Ramos-Rivera{\small $^{d}$},\
Primo\v z \v Sparl{\small $^{d,e,f,*}$}\ and 
Stephen E.~Wilson{\small $^{g}$}
\\

\medskip
{\it {\small
$^a$University of Alaska Fairbanks, Department of Mathematics and Statistics, Fairbanks, AK, USA\\
$^b$National Autonomous University of Mexico, Institute of Mathematics, Mexico City, Mexico\\
$^c$Northeastern University, Department of Mathematics, Boston, MA, USA\\
$^d$Institute of Mathematics, Physics and Mechanics, Ljubljana, Slovenia\\
$^e$University of Ljubljana, Faculty of Education, Ljubljana, Slovenia\\
$^f$University of Primorska, Institute Andrej Maru\v si\v c, Koper, Slovenia\\
$^g$Northern Arizona University, Department of Mathematics and Statistics, Flagstaff, AZ, USA
}}
\end{center}

\blfootnote{
Email addresses: 
lwberman@alaska.edu (L.~W.~Berman),
hiroki.koike@im.unam.mx (H.~Koike),
j.mochanquesnel@northeastern.edu (E.~Mochan),
alejandra.ramosrivera@fmf.uni-lj.si (A.~Ramos Rivera),
primoz.sparl@pef.uni-lj.si (P. \v Sparl),
stephen.wilson@nau.edu (S.~E.~Wilson)
\\
* - corresponding author
}

%%%%%%%%%%%%%%%%%
%%%		Abstract
%%%%%%%%%%%%%%%%%

\hrule

\begin{abstract}
A graph is {\em edge-transitive} if the natural action of its automorphism group on its edge set is transitive. An automorphism of a graph is {\em semiregular} if all of the orbits of the subgroup generated by this automorphism have the same length. While the tetravalent edge-transitive graphs admitting a semiregular automorphism with only one orbit are easy to determine, those that admit a semiregular automorphism with two orbits took a considerable effort and were finally classified in 2012. Of the several possible different ``types'' of potential tetravalent edge-transitive graphs admitting a semiregular automorphism with three orbits, only one ``type'' has thus far received no attention. In this paper we focus on this class of graphs, which we call the {\em Woolly Hat } graphs. We prove that there are in fact no edge-transitive Woolly Hat graphs and classify the vertex-transitive ones. 
\end{abstract}
\hrule

\begin{quotation}
\end{quotation}

%%%%%%%%%%%%%%%%%%%%%%%%%%%%%%
%----------------Introduction---------------
%%%%%%%%%%%%%%%%%%%%%%%%%%%%%%
\section{Introduction}
\label{sec:Intro}

Investigation of tetravalent graphs admitting a large degree of symmetry has been an active topic of research attracting many researchers as is witnessed by numerous publications on the subject. Among such graphs, the {\em edge-transitive} ones (for which the automorphism group acts transitively on the edge set of the graph) have received by far the most attention. An edge-transitive tetravalent graph $\G$ can be {\em arc-transitive} (the automorphism group $\Aut(\G)$ acts transitively on the set of all {\em arcs}, that is ordered pairs of adjacent vertices). If, however, $\G$ is not arc-transitive (but still is edge-transitive), then it is either {\em half-arc-transitive} or {\em semisymmetric}, depending on whether $\Aut(\G)$ acts transitively on the vertex set $V(\G)$ of $\G$ or not, respectively (see for instance~\cite{PotWil14, Spa09}).

The tetravalent edge-transitive graphs of each of these three ``types'' have been thoroughly investigated, but these families of graphs are simply far too rich and diverse to admit a complete classification. There is thus still an abundance of questions to be answered and problems to be solved. In attacking such a wide-ranging topic, it makes sense to investigate and possibly classify certain subfamilies of these graphs. When attempting to solve any such problem or decide what kind of subfamilies to study, it is beneficial to have a large database of tetravalent edge-transitive graphs. With this aim in mind, Poto\v cnik and Wilson constructed a Census~\cite{PotWil20} of all known tetravalent edge-transitive graphs up to order 512 (known to be complete for the arc-transitive and half-arc-transitive graphs but possibly incomplete for the semisymmetric ones) where for each graph in the Census, various properties of the graph are given, together with several known ways on how to construct it. 

To be able to describe the graphs from the Census and more generally within the whole family of tetravalent edge-transitive graphs it is desirable to have a large list of known infinite families of such graphs. One very natural way of obtaining such families is to analyze and possibly classify the examples admitting a {\em semiregular} automorphism (an automorphism having all orbits of equal length) with few orbits. It is an easy exercise to classify all tetravalent edge-transitive graphs admitting an automorphism with just one orbit (such graphs are known as {\em circulants}). The situation becomes much more interesting for graphs admitting a semiregular automorphism with two orbits (these are called {\em bicirculants}). It has taken a considerable effort of several researchers in a handful of papers before the last step of the classification was obtained in~\cite{KovKuzMalWil12}. For instance, one of the major steps was the classification of edge-transitive Rose window graphs~\cite{KovKutMar10}, a family of graphs introduced in~\cite{Wil08} that have since been the object of investigation in various contexts (see for instance~\cite{HubRamSpa21}). 

While all of the examples in the case of circulants and bicirculants turn out to be arc-transitive, the situation gets even more interesting when one considers the tetravalent edge-transitive {\em tricirculants}, that is, graphs admitting a semiregular automorphism with three orbits. Here, for the first time, half-arc-transitive examples start to appear. In fact, the smallest half-arc-transitive graph, the Doyle-Holt graph (see~\cite{AlsMarNow94, Doy76, Hol81}) turns out to be a tricirculant. Moreover, infinite families of tetravalent half-arc-transitive tricirculants have been constructed (see for instance~\cite{AlsMarNow94}). 

To be able to explain why we focus on a particular family of tetravalent tricirculants in this paper we introduce a notion of a diagram and its simplifed version. We point out that our diagrams are essentially what are known as voltage graphs and the corresponding graphs are what are known as cyclic covers of the corresponding simplified diagrams (see for instance~\cite{GroTuc87, MalNedSko00}). 

A tetravalent tricirculant can succinctly be represented by a {\em diagram} in the following way. Let $\rho$ be a corresponding semiregular automorphism, let $n$ be its order, let $u$, $v$ and $w$ be representatives of the three orbits of $\la \rho \ra$ and let us denote these orbits by $\mathcal{U}$, $\mathcal{V}$ and $\mathcal{W}$, respectively. We represent each of these three orbits by a circle, and then for each pair of distinct $x,y \in \{u,v,w\}$ we join the circle representing the orbit $\mathcal{X}$ of $x$ to the circle representing the orbit $\mathcal{Y}$ of $y$ by $k$ edges, where $k$ is the number of neighbors of $x$ within $\mathcal{Y}$. We orient all of these edges in one of the two possible ways, say from $\mathcal{X}$ to $\mathcal{Y}$, and for each $a$, $0 \leq a < n$, such that $x$ is adjacent to $y{\rho^a}$, we put the label $a$ onto this set of edges (where we usually omit the label $0$). For each $x \in \{u,v,w\}$ we also put a semiedge, a loop, or a semiedge and a loop at $\mathcal{X}$, whenever $x$ has $1$, $2$ or $3$ neighbors within $\mathcal{X}$, respectively. Note that a semiedge will be present at the circle representing $\mathcal{X}$ if and only if $n$ is even and $x$ is adjacent to $x\rho^{n/2}$. On a loop we put the label $a$, where $0 \leq a < n/2$ is such that $x$ is adjacent to $x{\rho^{a}}$ (and $x{\rho^{-a}}$). To indicate that $\rho$ is of order $n$ we also put an $n$ in the middle of the diagram. An example of a diagram is given in the left part of Figure~\ref{fig:diagram} (where we in addition indicate the names of the vertices from each of the three orbits as chosen in Construction~\ref{con:WHAT}). The {\em simplified diagram} is obtained from the diagram by omitting $n$ and the labels and arrows on the edges.
%, omitting potential loops but instead indicating the indegree of each of the orbits of $\la \rho \ra$ by a corresponding number within the circle representing that orbit. 

Distinguishing cases depending on the maximal {\em indegree} of the three orbits of $\la \rho \ra$ (the valency of the subgraph induced on that orbit), it is easy to see that there are $9$ possible simplified diagrams for tetravalent tricirculants (depicted in Figure~\ref{fig:tric_diag}). Considering the number of $4$-cycles through each edge, it is not difficult to show that none of the diagrams with an orbit of indegree $3$ (the first three) can give rise to edge-transitive graphs. Similar considerations reveal that the simplified diagram~IV gives rise to a unique edge-transitive example (it is the graph of order $18$, named C4[18,2] in~\cite{PotWil20}, which in fact also corresponds to the simplified diagram~IX, of course for a different semiregular automorphism of order $6$). The simplified diagram~VIII takes a bit more work, but using the results of~\cite{PotWil07} one can show that this diagram also gives rise to a unique edge-transitive example, namely the so called wreath graph of order $6$ (or equivalently the complement of the complete multipartite graph $K_{2,2,2}$). 
\begin{figure}[!h]
\begin{center}
	\includegraphics[scale=0.5]{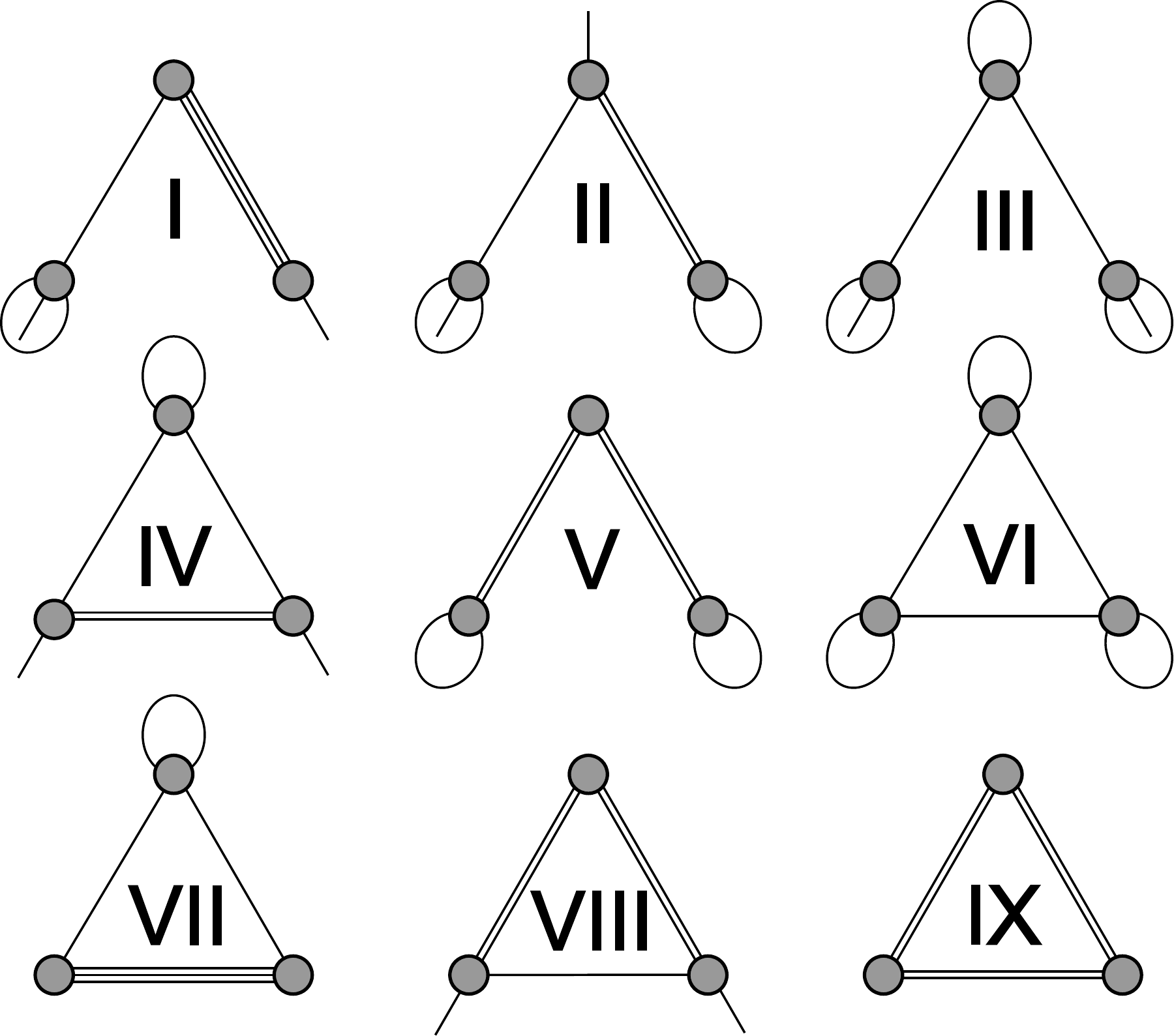}
	\caption{The nine simplified diagrams for tetravalent tricirculants.}
	\label{fig:tric_diag}
\end{center}
\end{figure}

The results of~\cite{Ste16} show that there are infinitely many edge-transitive examples for the simplified diagram~V, while those of~\cite{Mar98, MarSpa08, Spa09} show that there are infinitely many edge-transitive (in fact, infinitely many half-arc-transitive and also infinitely many arc-transitive) examples for each of the simplified diagrams~VI and~IX. 

This leaves us with the simplified diagram~VII. We call the corresponding diagram the {\em Woolly Hat diagram}. The corresponding graphs (see Section~\ref{sec:WHs}) have thus far not been investigated. It is the aim of this paper to investigate their symmetries. One of the main results of this paper (see Theorem~\ref{the:ET}) is that the Woolly Hat diagram yields no edge-transitive examples. 

When constructing semisymmetric tetravalent graphs, certain structures, called LR-structures (see~\cite{PotWil14, PotWil16, PotWil18}), can be very useful. The underlying graphs of these structures are tetravalent vertex-transitive graphs, which are not arc-transitive, but have some additional properties regarding their symmetries (see~\cite{PotWil14} for a definition). Inspecting the above mentioned Census~\cite{PotWil20}, we find out that some semisymmetric graphs from the Census arise via LR-structures corresponding to the Woolly Hat diagram. This motivates the second part of this paper, in which we give a complete classification of the vertex-transitive graphs corresponding to the Woolly Hat diagram (see Theorem~\ref{the:mainVT}).

%%%%%%%%%%%%%%%
\section{The family and basic properties}
\label{sec:WHs}

In this section we introduce the family of graphs corresponding to the Woolly Hat diagram and state some of their basic properties (but see also~\cite{PotWil20} where the graphs were first defined). Before doing this we would like to point out that throughout the paper $\ZZ_n$ will always denote the residue class ring modulo $n$, where $n$ is a positive integer.

\begin{construction}\label{con:WHAT}
For each integer $n \geq 3$, for each $a,b,c,d \in \ZZ_n$, where $2a \neq 0$, $b,c,d$ are pairwise distinct and $n$, $a$, $b$, $c$ and $d$ have no common prime divisor, the {\em Woolly Hat graph} $\WH_n(a,b,c,d)$ is the graph of order $3n$ with vertex set  
$$
	\{A_i \colon i \in \ZZ_n\} \cup \{B_i \colon i \in \ZZ_n\} \cup \{C_i \colon i \in \ZZ_n\}
$$
and in which for each $i \in \ZZ_n$ we have the following adjacencies
\begin{align*}
	A_i \sim\ & A_{i -a}, A_{i+a}, B_i, C_i\\
	B_i \sim\ & A_i, C_{i+b}, C_{i+c}, C_{i+d}\\
	C_i \sim\ & A_i, B_{i-b}, B_{i-c}, B_{i-d},\\
\end{align*}
where all of the subscripts are computed modulo $n$. The Wooly Hat diagram and a corresponding drawing of the graph $\WH_4(1,0,1,3)$ are given in Figure~\ref{fig:diagram}. 
\end{construction}

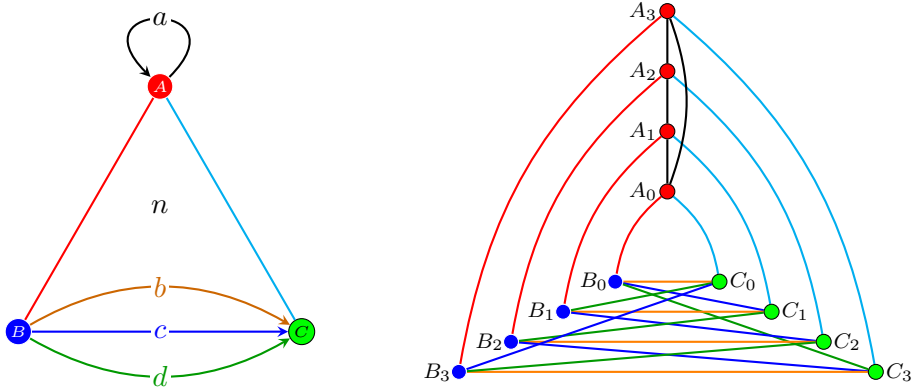
\begin{figure}[htbp]
\begin{center}
\subfigure
{
\begin{tikzpicture}[scale = 1.45]
\node[vtx, white, fill=red] (A) at (90: 1.5) {$A$};
\node[vtx, white, fill = blue] (B) at (90+120: 1.5) {$B$};
\node[vtx, , fill = green, ] (C) at (90+240: 1.5) {$C$};
\draw[thick, ->] (A) to[in = 90+45, out = 45,looseness = 3, distance = 1 cm] node[lbl]{$a$} %{$1$} 
(A);
\draw[thick, red] (A) -- (B);
\draw[thick, ->, orange!80!black] (B) to[bend left = 30] node[lbl] {$b$}%{$0$} 
(C); 
\draw[->, thick, blue] (B) to[bend left = 0] node[lbl] {$c$}%{$1$} 
(C); 
\draw[thick, ->, green!60!black] (B) to[bend left = -30] node[lbl] {$d$}%{$3$} 
(C); 
\draw[thick, cyan] (A) -- (C);
%\node[ right = 2mm of C] {$\mathbb{Z}_{n}$};
\node[ below = 12mm of A] {$n$};
\end{tikzpicture}
}
%%%%%%%%%%%%%%%%%%%
%\hspace{1cm}
\hspace{8mm}
%%%%%%%%%%%%%%%%%%%
\subfigure{
\begin{tikzpicture}[scale = .8]
\foreach \j in {0,1,2,3}{
 \pgfmathtruncatemacro{\rad}{\j+1};
\node[vtx, fill=red, inner sep = 2pt, label={[label distance=-3pt]left:{\scriptsize $A_{\j}$}}] (A\j) at (360*0/4+90:\rad) {};%{$\j$};%{$A_{\j}$};
\node[vtx, white, fill = blue, inner sep = 2pt, label={[label distance=-4pt]180:{\scriptsize $B_{\j}$}}]  (B\j) at (360*1/3+90: \rad){}; %{$B_{\j}$};
\node[vtx, , fill = green, inner sep = 2pt, label={[label distance=-3pt]0:{\scriptsize $C_{\j}$}}]  (C\j) at (360*2/3+90: \rad) {};%{$C_{\j}$};
}
\begin{scope}[on background layer]
\foreach \j in {0,1,2,3}{
%\draw[thick] let \n1 = {int(mod(\j+1, 4))} in (A\j) -- (A\n1);
\draw[thick,red] (A\j) to[bend right = 20] (B\j);
\draw[thick, orange] let \n1 = {int(mod(\j+0, 4))} in (B\j) to[bend right = 0] (C\n1); 
\draw[thick, blue] let \n1 = {int(mod(\j+1, 4))} in(B\j) -- (C\n1); 
\draw[thick, green!60!black] let \n1 = {int(mod(\j+3, 4))} in(B\j) -- (C\n1); 
\draw[thick,cyan] (A\j) to[bend left = 20](C\j);
}
\foreach \j in {0,1,2}{
\draw[thick] let \n1 = {int(mod(\j+1, 4))} in (A\j) -- (A\n1);
}
\draw[thick] (A3) to[bend left = 20] (A0);
\end{scope}
%\begin{scope}[on background layer]
%\draw[dotted, shorten >= -.5cm, shorten <= -.5cm] (A0.0) -- (A2.180);
%\end{scope}
%\node[ right = 2mm of C] {$\mathbb{Z}_{4}$};
\end{tikzpicture}
}
\caption{The Wooly Hat diagram and the graph $\WH_4(1, 0, 1, 3)$.}
\label{fig:diagram}
\end{center}
\end{figure}

%\begin{figure}[!h]
%\begin{center}
%	\includegraphics[scale=0.7]{WH_diag.pdf}
%	\caption{The diagram of WH-graphs and the example $\WH_4(1,0,1,3)$.}
%	\label{fig:diagram}
%\end{center}
%\end{figure}

Note that, by definition, the subscripts of the vertices of $\WH_n(a,b,c,d)$ are always computed modulo $n$. Note also that the assumptions that $2a \neq 0$ in $\ZZ_n$ and that $b,c$ and $d$ are pairwise distinct ensure that the graph is regular of valency $4$, while the assumption that $n$, $a$, $b$, $c$ and $d$ have no common prime divisor ensures that the graph is connected. 

Throughout the paper we abbreviate the term ``a Woolly Hat graph'' to {\em a WH-graph}. We also make an agreement that the vertices of a WH-graph of the form $A_i$, $B_i$ and $C_i$ are called the {\em $A$-vertices}, the {\em $B$-vertices} and the {\em $C$-vertices}, respectively. 

We first record some fairly obvious symmetries of the WH-graphs. The proof is straightforward and is left to the reader.

\begin{lemma}
\label{le:sym}
Let $\G = \WH_n(a,b,c,d)$ be a WH-graph. Then $\G$ admits the semiregular automorphism $\rho$, mapping according to the rule
\begin{equation}
\label{eq:rho}
A_i \rho = A_{i+1},\quad B_i\rho = B_{i+1},\quad C_i\rho = C_{i+1},\quad i \in \ZZ_n,
\end{equation}
and the automorphism $\tau$, mapping according to the rule
\begin{equation}
\label{eq:tau}
A_i \tau = A_{-i},\quad B_i\tau = C_{-i},\quad C_i\tau = B_{-i}, \quad i \in \ZZ_n.
\end{equation}
\end{lemma}

In Figure~\ref{fig:sym_draw} a symmetric drawing of the graph $\WH_4(1,0,1,3)$ from Figure~\ref{fig:diagram} is given in which the automorphism $\rho$ corresponds to a $4$-fold rotation while the automorphism $\tau$ corresponds to reflection over the dotted line. We remark that in Figures~\ref{fig:diagram} and \ref{fig:sym_draw}, as well as in Figures~\ref{fig:VT2} and~\ref{fig:sporadic} from Section~\ref{sec:VT}, the vertices and edges are colored in such a way that the vertices (or edges) from the same $\langle \rho \rangle$-orbit receive the same color. The sole purpose of this coloring is to make it easier to keep track of vertices and edges in these figures and should not be confused with the coloring of the edges we will be talking about in Section~\ref{sec:VT}.

\begin{figure}[htbp]
\begin{center}
\begin{tikzpicture}[scale = .8]
\foreach \j in {0,1,2,3}{
\pgfmathsetmacro{\aa}{360*\j/4-22.5}
\pgfmathsetmacro{\aaa}{360*\j/4+22.5}
\node[vtx, fill=red, inner sep = 2pt, label=360*\j/4:{\scriptsize $A_{\j}$}] (A\j) at (360*\j/4: 4) {};%{$A_{\j}$};
\node[vtx, white, fill = blue,inner sep = 2pt, label={[label distance = -5pt]\aa:{\scriptsize $B_{\j}$}}] (B\j) at (360*\j/4-22.5: 2) {};%{$B_{\j}$};
\node[vtx, , fill = green,inner sep = 2pt, label={[label distance = -5pt]\aaa:{\scriptsize $C_{\j}$}} ] (C\j) at (360*\j/4+ 22.5: 2) {};%{$C_{\j}$};
}
\foreach \j in {0,1,2,3}{
\draw[thick] let \n1 = {int(mod(\j+1, 4))} in (A\j) to[bend right = 5] (A\n1);
\draw[thick,red] (A\j) -- (B\j);
\draw[thick, orange] let \n1 = {int(mod(\j+0, 4))} in (B\j) --  (C\n1); 
\draw[thick, blue] let \n1 = {int(mod(\j+1, 4))} in(B\j) -- (C\n1); 
\draw[thick, green!60!black] let \n1 = {int(mod(\j+3, 4))} in(B\j) -- (C\n1); 
\draw[thick,cyan] (A\j) -- (C\j);
}
\begin{scope}[on background layer]
\draw[dotted, shorten >= -.5cm, shorten <= -.5cm] (A0.0) -- (A2.180);
\end{scope}
%\node[ right = 2mm of C] {$\mathbb{Z}_{4}$};
\end{tikzpicture}
\caption{A symmetric drawing of the graph $\WH_4(1, 0, 1, 3)$.}
\label{fig:sym_draw}
\end{center}
\end{figure}

Note that $\WH_n(a,b,c,d) = \WH_n(-a,b,c,d)$ and that for $\G = \WH_n(a,b,c,d)$ and any $q \in \ZZ_n$, coprime to $n$, relabeling each $A_i$, $B_i$ and $C_i$ by $A_{qi}$, $B_{qi}$ and $C_{qi}$, respectively, yields $\WH_n(qa,qb,qc,qd)$. This shows that the following holds.

\begin{lemma}
\label{le:iso}
Let $\G = \WH_n(a,b,c,d)$ be a WH-graph and let $q \in \ZZ_n$ be coprime to $n$. Then $\G \cong \WH_n(qa,qb,qc,qd)$. In particular, $\G \cong \WH_n(a,-b,-c,-d)$.
\end{lemma}

Let $\G$ be a WH-graph and let $\rho \in \Aut(\G)$ be as in~\eqref{eq:rho}. The subgroup $\la \rho \ra$ of $\Aut(\G)$ then has six orbits in its action on the edge set of $\G$. We call the edges from the $\la \rho \ra$-orbit of $A_0A_a$ the {\em $a$-edges}, those from the orbit of $A_0B_0$ the {\em left edges}, those from the orbit of $A_0C_0$ the {\em right edges}, and for each $x \in \{b,c,d\}$ the edges from the orbit of $B_0C_x$ the {\em $x$-edges}.

We also point out that each WH-graph has certain $6$-cycles. These will play an important role in the determination of the automorphism group of the WH-graphs. For each $i \in \ZZ_n$ we clearly have the $6$-cycles (recall that $b, c$ and $d$ are pairwise distinct)
\begin{equation}
\label{eq:can6} 
(B_i, C_{i+b},B_{i+b-c},C_{i+b-c+d},B_{i-c+d},C_{i+d}).
\end{equation}
We call the $6$-cycles of this form the {\em canonical} $6$-cycles. Observe that these $6$-cycles are characterized by the fact that for any pair of antipodal edges on such a $6$-cycle there exists an $x \in \{b,c,d\}$ such that both of these edges are $x$-edges.

%%%%%%%%%%%%%%%
\section{Nonexistence of edge-transitive WH-graphs}
\label{sec:ET}

In this section we prove Theorem~\ref{the:ET} which states that there are no edge-transitive WH-graphs. One of the main tools in our proof is the following nice result of A.~Lucchini, which can be found in the book~\cite{Isaacs_book} (we point out that the result was discovered independently by Herzog and Kaplan, see~\cite{HerKap01}), and has been used successfully several times before in various classification results concerning symmetries of graphs (for instance in~\cite{AntHujKut15, FreKut13, KovKutMar10, KovKutMarWil12, KovMarMuz13} to name just a few).

\begin{proposition}\cite{HerKap01, Isaacs_book}
\label{pro:lucchini}
Let $G$ be a group containing a nontrivial cyclic subgroup $H$. If $|G| \leq |H|^2$, then $H$ contains a nontrivial subgroup which is normal in $G$. 
\end{proposition}

We will be applying the above result in the setting where $G$ is the automorphism group of a WH-graph and $H$ is the cyclic subgroup generated by the semiregular automorphism $\rho$ from \eqref{eq:rho}. Before stating and proving the next proposition we recall the definition of a $2$-arc-transitive graph and define the notion of a quotient graph that we will be using in this paper. 

A $2$-arc of a graph is an ordered triple $(u,v,w)$ of its vertices such that $u \neq w$ and $v$ is adjacent to $u$ and $w$. A graph is {\em $2$-arc-transitive} if its automorphism group acts transitively on the set of its $2$-arcs. 
Let $\G$ be a graph and $K \leq \Aut(\G)$. The {\em quotient graph} $\G_K$ of $\G$ {\em with respect to} $K$ is the graph whose vertex set consists of the orbits of $K$ in its action on the vertex set of $\G$ and in which two distinct orbits $\mathcal{O}_1$ and $\mathcal{O}_2$ are adjacent whenever there are $u \in \mathcal{O}_1$ and $v \in \mathcal{O}_2$ which are adjacent in $\G$. Therefore, such a quotient graph is by definition simple (in the sense that there are no loops or multiple edges). 

\begin{proposition}
\label{pro:quotient}
Let $\G = \WH_n(a,b,c,d)$ be a WH-graph and let $\rho \in \Aut(\G)$ be as in \eqref{eq:rho}. Suppose $\G$ is edge-transitive and $K \leq \la \rho \ra$ is a nontrivial normal subgroup of $\Aut(\G)$. Then the quotient graph $\G_K$ of $\G$ with respect to $K$ is also an edge-transitive WH-graph. Moreover, if $\G$ is $2$-arc-transitive, then $\G_K$ is also $2$-arc-transitive.
\end{proposition}

\begin{proof}
Since $K$ is normal in $\Aut(\G)$, its orbits are blocks of imprimitivity for $\Aut(\G)$, and so edge-transitivity implies that there are no edges of $\G$ joining two vertices from the same $K$-orbit. Moreover, as $B_0$ is clearly the only neighbor of $A_0$ in the $K$-orbit of $B_0$, edge-transitivity also implies that for each vertex of $\G$ its four neighbors are in four different $K$-orbits. Therefore, the quotient graph $\G_K$ is regular of valency $4$. Since $K \leq \la \rho \ra$ it is now clear that $\G_K$ is a WH-graph. Moreover, the induced action of $\Aut(\G)$ on $\G_K$ is edge-transitive, and if $\Aut(\G)$ acts $2$-arc-transitively on $\G$, then so does this induced action of $\Aut(\G)$ on $\G_K$. 
\end{proof}

Our proof that there are no edge-transitive WH-graphs is carried out in two steps. We first show that there are no $2$-arc-transitive WH-graphs and then show that there are also no edge-transitive ones. But first we make a useful observation.

\begin{lemma}
\label{le:girth3}
There are no edge-transitive WH-graphs of girth $3$. Consequently, if the graph  $\WH_n(a,b,c,d)$ is edge-transitive, then $b$, $c$ and $d$ are all nonzero. 
\end{lemma}

\begin{proof}
Suppose $\G = \WH_n(a,b,c,d)$ is edge-transitive and is of girth $3$. Let $x \in \{b,c,d\}$. Since $\G$ is edge-transitive, the edge $B_0C_x$ is contained in a $3$-cycle, and so $B_0$ and $C_x$ have a common neighbor. By definition of $\G$ this is only possible if $x = 0$ (and this common neighbor is $A_0$). But then $b = c = d = 0$, contradicting the assumption from Construction~\ref{con:WHAT} that $b$, $c$ and $d$ are pairwise distinct.  
\end{proof}

Using Lemma~\ref{le:iso} by which we can assume that $a$ divides $n$, together with Lemma~\ref{le:girth3}, a computer search for all relatively small edge-transitive WH-graphs can be performed. A couple of hours of computation using a suitable package such as {\sc Magma}~\cite{Mag} on an ordinary computer reveals that there are no edge-transitive WH-graphs of order up to $3\cdot 71$ (that is, with $n \leq 71$). For ease of reference we record this fact as an observation.

\begin{observation}
\label{obs:PC}
There exists no edge-transitive WH-graph $\WH_n(a,b,c,d)$ with $n \leq 71$.
\end{observation}

\begin{proposition}
\label{pro:2AT}
There are no $2$-arc-transitive WH-graphs.
\end{proposition}

\begin{proof}
By way of contradiction, suppose $2$-arc-transitive WH-graphs exist and let $\G = \WH_n(a,b,c,d)$ be a smallest $2$-arc-transitive WH-graph. Let $\rho$ be as in \eqref{eq:rho}. By Proposition~\ref{pro:quotient}, no nontrivial subgroup of $\la \rho \ra$ is normal in $\Aut(\G)$, and so Proposition~\ref{pro:lucchini} implies that $|\Aut(\G)| > n^2$. Since $\G$ is of order $3n$ and is $2$-arc-transitive, $|\Aut(\G)| = 3n\cdot 12\cdot t$, where $t$ is the order of a stabilizer of a $2$-arc in $\Aut(\G)$. Then $n < 36t$, and so Observation~\ref{obs:PC} implies that $t > 2$. Analysing $6$-cycles of $\G$ we now show that this is not possible.

Since $\G$ is $2$-arc-transitive and possesses the canonical $6$-cycles, there exists a positive constant $\lambda$ such that each $2$-arc of $\G$ lies on $\lambda$ different $6$-cycles. Clearly, any $6$-cycle of $\G$ containing the $2$-path $P = (B_0,A_0,C_0)$ is of the form 
\begin{equation}
\label{eq:6cyc_trian}
(C_x,B_0,A_0,C_0,B_{-y},A_{-y}),
\end{equation} 
where $x,y \in \{b,c,d\}$ are such that $x = -y$. Therefore, the above $2$-path $P$ lies on a $6$-cycle of $\G$ if and only if $2x = 0$ for some $x \in \{b,c,d\}$ or $x+y = 0$ for a pair of distinct $x,y \in \{b,c,d\}$ (or both). Since $b,c$ and $d$ are pairwise distinct and are all nonzero (by Lemma~\ref{le:girth3}), $2x = 0$ for at most one $x \in \{b,c,d\}$ and $x+y = 0$ for at most one pair of distinct $x,y \in \{b,c,d\}$. Therefore, since a condition of the form $2x = 0$ gives rise to one $6$-cycle through $P$, while a condition of the form $x+y = 0$ for $x \neq y$ gives rise to two $6$-cycles through $P$, we have that $\lambda \in \{1,2,3\}$. 

Suppose first that $\lambda = 1$. Inspecting the canonical $6$-cycles from \eqref{eq:can6} we see that in this case, for any $2$-path containing no $A$-vertices, the unique $6$-cycle through it is a canonical $6$-cycle. It thus follows that the unique $6$-cycle through the $2$-path $(A_0,A_a,C_a)$ must contain a $4$-path of the form $(A_0,A_a,C_a,B_{a-x},A_{a-x})$, where $x \in \{b,c,d\}$. Therefore, $A_0$ and $A_{a-x}$ must have a common neighbor, which obviously must be an $A$-vertex and is thus $A_{-a}$ (and so $a-x = -2a$). But then this $6$-cycle and the one obtained by applying $\rho^a$ are two $6$-cycles containing the $2$-path $(A_{-a},A_0,A_a)$, contradicting $\lambda = 1$.

Suppose next that $\lambda = 3$. The above discussion shows that we can assume that $b = n/2$ and $c+d = 0$. Consider the corresponding three $6$-cycles from \eqref{eq:6cyc_trian}. The successors of $B_0$ (in the direction of the $2$-arc $(C_0,A_0,B_0)$) on them are $C_b$, $C_c$ and $C_d$, and so each $3$-arc whose initial $2$-arc is $(C_0,A_0,B_0)$ lies on a $6$-cycle. As $\G$ is $2$-arc-transitive, each $3$-arc lies on a $6$-cycle. But this clearly is not the case for the $3$-arc $(B_0,A_0,A_a,B_a)$, since no two $C$-vertices are adjacent, showing that $\lambda \neq 3$.

We are left with the possibility $\lambda = 2$. Recall that in this case we can assume that $c+d = 0$ and $b \neq n/2$. Again, inspecting the corresponding $6$-cycles from \eqref{eq:6cyc_trian} we find that for each $2$-arc $(u,v,w)$ of $\G$ there is a unique $z \in \G(w) \setminus \{v\}$ (here $\G(w)$ is the set of all neighbors of $w$ in $\G$) such that $\G$ has no $6$-cycle containing the $3$-arc $(u,v,w,z)$, and moreover, that the two $6$-cycles of $\G$ through $(u,v,w)$ have precisely these three vertices in common. We now show that this contradicts $t > 2$. To this end, let $(u,v,w)$ be any $2$-arc of $\G$, let $x \in \G(u) \setminus \{v\}$ and $y \in \G(w) \setminus \{v\}$ be the unique vertices such that $\G$ has no $6$-cycle through $(x,u,v,w)$ or $(u,v,w,y)$, and let $u',u'', w',w''$ be such that $\G(u) = \{x,u',u'',v\}$, $\G(w) = \{v,w',w'',y\}$ and that $\G$ has a $6$-cycle through each of $(u', u, v, w, w')$ and $(u'', u, v, w, w'')$, see Figure~\ref{fig:2ATproof}. 
%\begin{figure}[!h]
%\begin{center}
%	\includegraphics[scale=1]{fig_proof_2AT.pdf}
%	\caption{The local situation around a $2$-arc in the case of $\lambda = 2$.}
%	\label{fig:2ATproof}
%\end{center}
%\end{figure}
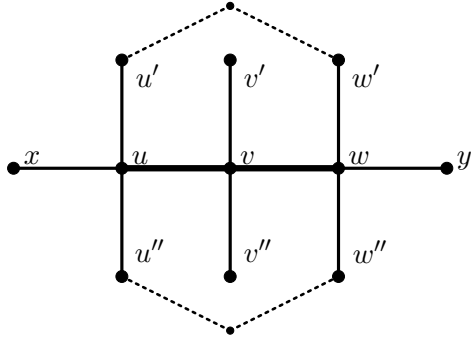
\begin{figure}[htbp]
\begin{center}
\begin{tikzpicture}[line cap=round,line join=round,>=triangle 45,x=.8cm,y=.8cm, scale = 0.9]
\clip(5,8.9) rectangle (14.5,15.5);
\draw [line width=1.2pt] (6,12)-- (8,12);
\draw [line width=2.4pt] (8,12)-- (10,12);
\draw [line width=2.4pt] (10,12)-- (12,12);
\draw [line width=1.2pt] (12,12)-- (14,12);
\draw [line width=1.2pt] (8,12)-- (8,14);
\draw [line width=1.2pt] (10,12)-- (10,14);
\draw [line width=1.2pt] (12,12)-- (12,14);
\draw [line width=1.2pt] (12,12)-- (12,10);
\draw [line width=1.2pt] (10,12)-- (10,10);
\draw [line width=1.2pt] (8,12)-- (8,10);
\draw [line width=1pt,dotted] (8,14)-- (10,15);
\draw [line width=1pt,dotted] (10,15)-- (12,14);
\draw [line width=1pt,dotted] (8,10)-- (10,9);
\draw [line width=1pt,dotted] (10,9)-- (12,10);
\draw (8.051534646409303,10.8683662429718) node[anchor=north west] {$u''$};
\draw (10.063554021952399,10.8683662429718) node[anchor=north west] {$v''$};
\draw (12.075573397495495,10.849019902822349) node[anchor=north west] {$w''$};
\draw (6,12.5) node[anchor=north west] {$x$};
\draw (8.0,12.5) node[anchor=north west] {$u$};
\draw (12.0,12.5) node[anchor=north west] {$w$};
\draw (10.0,12.5) node[anchor=north west] {$v$};
\draw (14.0,12.5) node[anchor=north west] {$y$};
\draw (8.070880986558757,14.1) node[anchor=north west] {$u'$};
\draw (10.063554021952399,14.1) node[anchor=north west] {$v'$};
\draw (12.05622705734604,14.1) node[anchor=north west] {$w'$};
\begin{scriptsize}
\draw [fill=black] (6,12) circle (2.5pt);
\draw [fill=black] (8,12) circle (2.5pt);
\draw [fill=black] (10,12) circle (2.5pt);
\draw [fill=black] (12,12) circle (2.5pt);
\draw [fill=black] (14,12) circle (2.5pt);
\draw [fill=black] (8,14) circle (2.5pt);
\draw [fill=black] (10,14) circle (2.5pt);
\draw [fill=black] (12,14) circle (2.5pt);
\draw [fill=black] (8,10) circle (2.5pt);
\draw [fill=black] (10,10) circle (2.5pt);
\draw [fill=black] (12,10) circle (2.5pt);
\draw [fill=black] (10,15) circle (1.5pt);
\draw [fill=black] (10,9) circle (1.5pt);
\end{scriptsize}
\end{tikzpicture}
\caption{The local situation around a $2$-arc in the case of $\lambda = 2$.}
\label{fig:2ATproof}
\end{center}
\end{figure}
Since one of the two $6$-cycles through $(u',u,v)$ contains $(u',u,v,w)$, we can also assume that $\G(v) = \{u,v',v'',w\}$, where $v''$ is such that there are no $6$-cycles through $(u',u,v,v'')$ (note however that there is a $6$-cycle through $(u',u,v,v')$). Suppose now that $\alpha \in \Aut(\G)$ fixes each of $u$, $v$ and $w$. By definition of $x$ and $y$ it then also fixes each of $x$ and $y$. Now, suppose that in addition $\alpha$ also fixes at least one of $u',u'',v',v'',w',w''$. Because of the two $6$-cycles through $(u,v,w)$ it is clear that if $\alpha$ fixes any of $u',u'',w',w''$ it fixes each of them. Because of the two $6$-cycles through $(u',u,v)$ it now follows that if $\alpha$ fixes any of $u',u'',w',w''$ it then also fixes $v'$ and thus also $v''$. Conversely, if $\alpha$ fixes any of $v'$ and $v''$, it fixes both, and then the fact that $u' \in \G(u) \setminus \{v\}$ is the unique vertex such that there is no $6$-cycle through $(v'',v,u,u')$ implies that $\alpha$ also fixes $u'$ (and thus each of $u',u'',w',w''$). This thus shows that if $\alpha$ fixes any of $u',u'',v',v'',w',w''$ it fixes each of them. It is now easy to see that any such $\alpha$ also fixes each neighbor of any of $x,y,u',u'',v',v'',w',w''$. Since $\G$ is connected, an inductive approach then shows that $\alpha$ is the identity. Therefore, the stabilizer of $(u,v,w)$ in $\Aut(\G)$ has at most one nontrivial element, and so $t \leq 2$, a contradiction. 
\end{proof}

Before stating and proving the main result of this section we review a concept from~\cite{MikPotWil08} that will play an important role in our proof. Here we only describe the part essential to our proof, but see~\cite{MikPotWil08} for details. Suppose $\G$ is a tetravalent arc-transitive graph and $\cC$ is a set of cycles of $\G$ such that each edge of $\G$ lies on precisely one cycle from $\cC$ (in other words, $\cC$ is a decomposition of the edge set of $\G$ into cycles). If a subgroup $G \leq \Aut(\G)$ preserves the set $\cC$ (that is, each element of $G$ maps cycles from $\cC$ to cycles from $\cC$) then we say that $\cC$ is a {\em $G$-invariant cycle decomposition} of $\G$. Our argument in the proof of Theorem~\ref{the:ET} will rely on the following corollary of~\cite[Theorem~4.2]{MikPotWil08}.

\begin{proposition}\cite{MikPotWil08}
\label{pro:cycle_str}
Let $\G$ be a tetravalent arc-transitive graph. If $\G$ is not $2$-arc-transitive, then there exists at least one $\Aut(\G)$-invariant cycle decomposition of $\G$. 
\end{proposition}

\begin{theorem}
\label{the:ET}
There are no edge-transitive WH-graphs.
\end{theorem}

\begin{proof}
We first point out that since each WH-graph admits the automorphism $\tau$ from Lemma~\ref{le:sym} a WH-graph is edge-transitive if and only if it is arc-transitive. It thus suffices to show that there are no arc-transitive WH-graphs. By way of contradiction, suppose arc-transitive WH-graphs do exist and let $\G = \WH_n(a,b,c,d)$ be a smallest arc-transitive WH-graph. Proposition~\ref{pro:2AT} implies that $\G$ is not $2$-arc-transitive, and so Proposition~\ref{pro:cycle_str} implies that $\G$ admits at least one $\Aut(\G)$-invariant cycle decomposition. Let $\cC$ be one of them. For an adjacent pair of vertices $u$, $v$ of $\G$ we denote the unique member of $\cC$ containing the edge $uv$ by $\cC^{uv}$.

Just as in the proof of Proposition~\ref{pro:2AT} we obtain that $|\Aut(\G)| > n^2$. Since $\G$ is of order $3n$ and is arc-transitive, $|\Aut(\G)| = 3n\cdot 4\cdot t$, where $t$ is the order of a stabilizer of an arc in $\Aut(\G)$. Then $n < 12t$, and so Observation~\ref{obs:PC} implies that $t > 6$. In particular, $t > 1$. Now, let $u$ be a vertex of $\G$ and $v,v',w,w'$ be its four neighbors, where we have denoted them in such a way that $(v,u,v')$ is contained in $\cC^{uv}$ and $(w,u,w')$ is contained in $\cC^{uw}$. Let $s$ be the size of the intersection of the vertex sets of $\cC^{uv}$ and $\cC^{uw}$. Since $t > 1$, there exists an automorphism $\alpha \in \Aut(\G)$ fixing $v$ and $u$ while moving at least one of $v',w$ and $w'$, but as $\cC$ is $\Aut(\G)$-invariant, $\alpha$ fixes $\cC^{uv}$ pointwise and reflects $\cC^{uw}$ with respect to $u$. Consequently, $\alpha$ fixes each of the $s$ common vertices of $\cC^{uv}$ and $\cC^{uw}$ while reflecting $\cC^{uw}$, and so $s \leq 2$ must hold. As $\G$ is vertex-transitive and $\cC$ is $\Aut(\G)$-invariant, each pair of distinct nondisjoint cycles of $\cC$ meets in $s$ vertices.

Consider now the left edge $A_0B_0$. With no loss of generality we can assume that the member of $\cC$ containing this edge contains the $2$-path $(A_0,B_0,C_c)$. Applying the automorphism $\tau\rho^c$, where $\rho$ and $\tau$ are from Lemma~\ref{le:sym}, we see that this cycle in fact contains the $3$-path $(A_0,B_0,C_c,A_c)$. Moreover, considering the orbits of this cycle under the subgroup $\la \rho\ra$ we see that the other member of $\cC$ containing the vertex $B_0$ consists solely of $b$- and $d$-edges and is thus of the form $(B_0,C_b,B_{b-d},C_{2b-d}, \ldots)$. Edge-transitivity of $\G$ therefore implies that all members of $\cC$ are of length 
\begin{equation}
\label{eq:cyc_length}
\frac{2n}{\gcd(d-b,n)}.
\end{equation}
We distinguish two possibilities depending on what the cycle $\cC^{A_0B_0}$ looks like.
\smallskip

\noindent
{\sc Case 1}: there is a cycle in $\cC$ containing $(B_0,A_0,C_0)$.\\
In view of the automorphism $\rho$ it is clear that, besides the cycles consisting solely of the $b$- and $d$-edges, $\cC$ consists of all the cycles of the forms
\begin{align*}
(\ldots, A_i,A_{i+a},A_{i+2a},\ldots), &\ i \in \ZZ_n\\
(\ldots, B_i, A_i, C_i, B_{i-c}, A_{i-c}, C_{i-c},B_{i-2c}, \ldots), &\ i \in \ZZ_n.
\end{align*}
The first of these are of length $n/\gcd(a,n)$ while the second are of length $3n/\gcd(c,n)$. As both of these lengths need to equal the length from~\eqref{eq:cyc_length}, we obtain $\gcd(d-b,n) = 2\gcd(a,n)$ and $\gcd(c,n) = 3\gcd(a,n)$. In particular, $\la c \ra$ is an index $3$ subgroup of $\la a \ra$ in $\ZZ_n$. The intersection of the two members of $\cC$ containing the vertex $A_0$ is thus 
$$
	\{A_i \colon i \in \la a \ra \cap \la c \ra\} = \{A_i \colon i \in \la c \ra\},
$$
and so $s \leq 2$ implies that $2c = 0$. Lemma~\ref{le:girth3} thus forces $n$ to be even and $c = n/2$. The cycles from $\cC$ are therefore of length $6$, $n$ is divisible by $6$, and the two members of $\cC$ containing $A_0$ meet in their antipodal pair of vertices. However, considering the two members of $\cC$ containing $B_0$ we see that the antipodal vertex in one of them is $B_{n/2}$, while in the other one it is a $C$-vertex, which contradicts vertex-transitivity of $\G$. 
\smallskip

\noindent
{\sc Case 2}: no cycle in $\cC$ contains $(B_0,A_0,C_0)$.\\
With no loss of generality we can then assume that the cycle $\cC^{A_0B_0}$ contains $(B_0,A_0,A_a)$, and so applying $\rho$ we see that besides the cycles consisting solely of the $b$- and $d$-edges, the only other members of $\cC$ are of the form
$$
	(\ldots, B_i,A_i,A_{i+a},C_{i+a},B_{i+a-c},A_{i+a-c},A_{i+2a-c},C_{i+2a-c},B_{i+2a-2c}, \ldots), \ i \in \ZZ_n.
$$
This implies that the two members of $\cC$ containing the vertex $B_0$ are of the form
\begin{align*}
	\cC^{B_0C_b} \colon & (\ldots , C_{2d-b}, B_{d-b}, C_d, B_0, C_b, B_{b-d}, \ldots ) \\
	\cC^{B_0C_c} \colon & (\ldots, C_a, A_a, A_0, B_0, C_c, A_c, \ldots).
\end{align*}
Since $t > 1$, there exists an automorphism $\alpha \in \Aut(\G)$ fixing each of $B_0$ and $C_b$ but interchanging $A_0$ and $C_c$. It follows that $\alpha$ fixes the cycle $\cC^{B_0C_b}$ pointwise and reflects $\cC^{B_0C_c}$ with respect to $B_0$. In particular, $\alpha$ fixes $B_{b-d}$ and interchanges $A_0$ with $C_c$, and thus also $\cC^{A_0C_0}$ with $\cC^{C_cB_{c-d}}$. Observe that these two cycles are of the form
\begin{align*}
	\cC^{A_0C_0} \colon & (\ldots , B_{-a}, A_{-a}, A_0, C_0, B_{-c}, \ldots ) \\
	\cC^{C_cB_{c-d}} \colon & (\ldots, C_{c-b+d}, B_{c-b}, C_c,B_{c-d},C_{b+c-d}, \ldots).
\end{align*}
Therefore, $\alpha$ maps $C_{b+c-d}$ to one of $B_{-a}$ and $B_{-c}$. However, the fixed vertex $B_{b-d}$ is adjacent to $C_{b+c-d}$ but is clearly not adjacent to any of $B_{-a}$ and $B_{-c}$, bringing us to the final contradiction.
\end{proof}

%%%%%%%%%%%%%%%
\section{The vertex-transitive WH-graphs}
\label{sec:VT}

While there are no edge-transitive WH-graphs, one can find vertex-transitive examples. In fact, as the next two propositions show, there are infinitely many of them. Before introducing the corresponding two infinite families we make a simple observation. Let $\G = \WH_n(a,b,c,d)$ be a WH-graph and recall that $\G$ admits the automorphisms $\rho$ and $\tau$ from Lemma~\ref{le:sym}. The subgroup $\la \rho, \tau \ra$ of $\Aut(\G)$ has two orbits on the vertex set of $\G$ with one orbit consisting of all the $A$-vertices. It thus follows that $\G$ is vertex-transitive if and only if there is an automorphism of $\G$ mapping at least one $A$-vertex to a $B$- or a $C$-vertex. 

\begin{proposition}
\label{pro:VTfam1}
Let $n \geq 4$ be an even integer and let $a,b,c,d$ be integers with $1 \leq a < n/2$ and $0 \leq b,c,d < n$ such that $b,c$ and $d$ are pairwise distinct, $a$, $b$ and $d$ are odd, while $c$ is even, and each of the following holds:
\begin{itemize}
\itemsep = 0pt
\item $d = 2a + b$;
\item $2c = 3a+3b$.
\end{itemize}
Then the graph $\WH_n(a,b,c,d)$ is vertex-transitive.
\end{proposition}

\begin{proof}
By the comment preceding the statement of this proposition we only need to exhibit an automorphism of $\G = \WH_n(a,b,c,d)$ mapping an $A$-vertex to a $B$- or a $C$-vertex. Let $\sigma$ be the permutation of the vertex set of $\G$ defined by the rule given via the following table, where $i \in \ZZ_n$:
$$
\begin{array}{c||c|c|c}
i\, \text{mod}\, 2 & A_i\sigma & B_i\sigma & C_i\sigma \\ \hline
0 & B_i & C_{i+c} & A_i \\
1 & C_{i-a+d} & A_{i+a+b} & B_{i-b+c-d}
\end{array}
$$
Due to the assumptions on the parity of $a, b, c$ and $d$ it is clear that $\sigma$ is indeed a permutation of the vertex set of $\G$. Using the assumption that $2a = d-b$ and $2c = 3a+3b$ one can also verify that it preserves adjacencies. For instance, while it is clear that the $a$-edges $A_iA_{i+a}$ with $i$ even are mapped to edges of $\G$, one can use the assumption $d = 2a+b$ to see that also for $i$ odd the vertex $A_{i+a}\sigma = B_{i+a}$ is adjacent to $A_i\sigma = C_{i-a+d} = C_{i+a+b}$. Similarly, to see that the right edges $A_iC_i$ with $i$ odd are mapped to edges of $\G$, one needs to show that $C_{i-a+d}$ and $B_{i-b+c-d}$ are adjacent. As $d = 2a+b$ and $2c = 3a+3b$, we see that $C_{i-a+d} = C_{i+a+b}$ and $B_{i-b+c-d} = B_{i-2a-2b+2c-c} = B_{i+a+b-c}$ are indeed adjacent. That $\sigma$ preserves all other adjacencies is verified in a similar way and is left to the reader.
\end{proof}

Figure~\ref{fig:VT2} shows symmetric drawings of the graphs $\WH_4(1,3,0,1)$ and $\WH_4(1,3,2,1)$, the smallest two members of the family of vertex-transitive WH-graphs from Proposition~\ref{pro:VTfam1}. The colors of the vertices are consistent with Figure~\ref{fig:diagram}. For each of these drawings the $6$-fold rotation is an automorphism mapping an $A$-vertex to a $B$-vertex. We remark that the counterclockwise rotation corresponds to the automorphism $\sigma \rho^2$ in the case of $\WH_4(1,3,0,1)$ and to the automorphism $\sigma$ in the case of $\WH_4(1,3,2,1)$, where $\sigma$ is as in the proof of Proposition~\ref{pro:VTfam1} and $\rho$ as in Lemma~\ref{le:sym}.

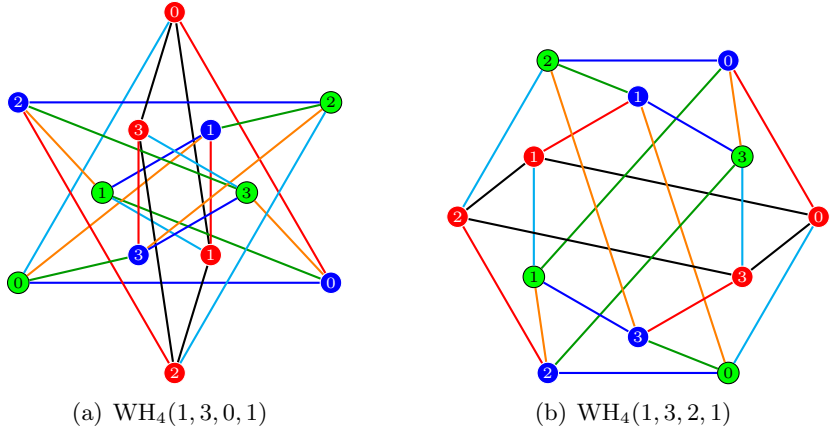
\begin{figure}[htbp]
\begin{center}
\subfigure[$\WH_4(1,3,0,1)$]{
\begin{tikzpicture}[scale = .8]

%$WH_{4}(1,3,0,1)$
%imod2 Ai? Bi? Ci?
%0 Bi Ci+c Ai
%1 Ci?a+d Ai+a+b Bi?b+c?d

% A1 - C1 - B3 - A3 - C3 - B1

%Cycles[{{aa[0], bb[2], cc[0], aa[2], bb[0], cc[2]}, {aa[1], cc[3], bb[1], aa[3], cc[1], bb[3]}}],

\pgfmathsetmacro{\rr}{3}
\pgfmathsetmacro{\rrr}{1.2}

\pgfmathsetmacro{\off}{90}
\node[vtx,white, fill=red] (A0) at (360*0/6+\off:\rr){0};
\node[vtx,white, fill = blue ] (B2) at (360*1/6+\off:\rr){2};
\node[vtx, , fill = green](C0) at (360*2/6+\off:\rr){0};
\node[vtx,white, fill=red] (A2) at (360*3/6+\off:\rr){2};
\node[vtx, white, fill = blue ] (B0) at (360*4/6+\off:\rr){0};
\node[vtx, , fill = green] (C2) at (360*5/6+\off:\rr){2};

\pgfmathsetmacro{\off}{-60}
\node[vtx,white, fill=red] (A1) at (360*0/6+\off:\rrr){1};
\node[vtx, , fill = green ] (C3) at (360*1/6+\off:\rrr){3};
\node[vtx, white, fill = blue](B1) at (360*2/6+\off:\rrr){1};
\node[vtx,white, fill=red] (A3) at (360*3/6+\off:\rrr){3};
\node[vtx, , fill = green ] (C1) at (360*4/6+\off:\rrr){1};
\node[vtx, white, fill = blue] (B3) at (360*5/6+\off:\rrr){3};

\foreach \j in {0,1,2,3}{
\draw[thick] let \n1 = {int(mod(\j+1, 4))} in (A\j) -- (A\n1);
\draw[thick,red] (A\j) -- (B\j);
\draw[thick, orange] let \n1 = {int(mod(\j+3, 4))} in (B\j) --  (C\n1); 
\draw[thick, blue] let \n1 = {int(mod(\j+0, 4))} in(B\j) -- (C\n1); 
\draw[thick, green!60!black] let \n1 = {int(mod(\j+1, 4))} in(B\j) -- (C\n1); 
\draw[thick,cyan] (A\j) -- (C\j);
}

%\node[ right = 2mm of C] {$\mathbb{Z}_{4}$};
\end{tikzpicture}}
\hspace{1cm}
\subfigure[$\WH_4(1,3,2,1)$]{
\begin{tikzpicture}[scale = .8]
%\foreach \j in {0,1,2,3}{
%\node[vtx, fill=red] (A\j) at (360*\j/4 : 3) {$A_{\j}$};
%\node[vtx, white, fill = blue] (B\j) at (360*\j/4+45: 3) {$B_{\j}$};
%\node[vtx, , fill = green, ] (C\j) at (360*\j/4+ 22.5 +180: 1) {$C_{\j}$};
%}

\foreach \j in {0,2}{
\node[vtx,white, fill=red] (A\j) at (180*\j/2 : 3) {\j};%{$A_{\j}$};
\node[vtx, white, fill = blue] (B\j) at (180*\j/2 + 60: 3){\j};% {$B_{\j}$};
\node[vtx, , fill = green, ] (C\j) at (180*\j/2+ 60 -120: 3) {\j};%{$C_{\j}$};
}

\foreach \j in {1,3}{
\node[vtx,white, fill=red]  (A\j) at (180*\j/2 + 120+120+180: 2) {\j};%{$A_{\j}$};
\node[vtx, , fill = green] (C\j) at (180*\j/2 + 60+ 60: 2){\j};% {$C_{\j}$};
\node[vtx, ,white, fill = blue, ] (B\j) at (180*\j/2+ 60 -120+ 60: 2) {\j};%{$B_{\j}$};
}

\foreach \j in {0,1,2,3}{
\draw[thick] let \n1 = {int(mod(\j+1, 4))} in (A\j) -- (A\n1);
\draw[thick,red] (A\j) -- (B\j);
\draw[thick, orange] let \n1 = {int(mod(\j+3, 4))} in (B\j) --  (C\n1); 
\draw[thick, blue] let \n1 = {int(mod(\j+2, 4))} in(B\j) -- (C\n1); 
\draw[thick, green!60!black] let \n1 = {int(mod(\j+1, 4))} in(B\j) -- (C\n1); 
\draw[thick,cyan] (A\j) -- (C\j);
}

%\node[ right = 2mm of C] {$\mathbb{Z}_{4}$};
\end{tikzpicture}}
%
%\subfigure[$WH_{12}(2,1,0,3)$]{
%%
%%i mod 4 	Ai? Bi? Ci? 
%%0 			Bi Ai Ci
%%? 			Ci Bi Ai 
%%2 			Ci+m Ai+m Bi+m
%%4 ? ? 		Bi?m Ci?m Ai?m
%
%%delta = 3, m = 3
%% (A0, B0)(A4,B4)(A8,B8),(C0)(C4)(C8)
%% (A3, C3)(A7, C7)(A11,C11) (B3)(B7)(B11)
%% (A2, C5)(A6, C9)(A10, C1) 
%% (A1, B10)(A5, B2)(A9, B6)(B1, C10)(B5, C2), (B9, C6)
%
%}

%Cycles[{{aa[0], bb[1], bb[2], aa[8], bb[9], bb[10], aa[16], bb[17], 
%   bb[18], aa[4], bb[5], bb[6], aa[12], bb[13], bb[14]}, 
%
%{aa[1], cc[2], bb[8], aa[9], cc[10], bb[16], aa[17], cc[18], bb[4], aa[5], 
%   cc[6], bb[12], aa[13], cc[14], bb[0]}, 
%
%{aa[2], cc[8], cc[9], 
%   aa[10], cc[16], cc[17], aa[18], cc[4], cc[5], aa[6], cc[12], 
%   cc[13], aa[14], cc[0], cc[1]}, 
%
%{aa[3], bb[19], cc[15], aa[11], 
%   bb[7], cc[3], aa[19], bb[15], cc[11], aa[7], bb[3], cc[19], aa[15],
%    bb[11], cc[7]}}],

%\subfigure[$WH_{20}(2,3,0,7)$]{
%
%blah
%}
\caption{The smallest two examples of the family of WH-graphs from Proposition~\ref{pro:VTfam1}.}
\label{fig:VT2}
\end{center}
\end{figure}

\begin{proposition}
\label{pro:VTfam2}
Let $m \geq 3$ be an odd integer. Then the WH-graph $\WH_{4m}(2,m-2,0,m+2)$ is vertex-transitive.
\end{proposition}

\begin{proof}
As in the previous proof, we only need to show that there exists an automorphism of $\G = \WH_{4m}(2,m-2,0,m+2)$ which maps an $A$-vertex to a $B$- or a $C$-vertex. Since $m$ is odd, there is a unique $\delta \in \{1,3\}$ such that $m \equiv \delta \pmod{4}$. Let $\theta$ be the permutation of the vertex set of $\G$ defined by the rule given via the following table, where $i \in \ZZ_n$:
$$
\begin{array}{c||c|c|c}
i\, \text{mod}\, 4 & A_i\theta & B_i\theta & C_i\theta \\ \hline
0 & B_i & A_i & C_i \\
\delta & C_i & B_i & A_i \\
2 & C_{i+m} & A_{i+m} & B_{i+m} \\
4-\delta & B_{i-m} & C_{i-m} & A_{i-m}
\end{array}
$$
Using the fact that $m$ is odd and $m \equiv \delta \pmod{4}$, it is easy to verify that $\theta$ is indeed a permutation of the vertex set of $\G$. Since $c = 0$, it is clear that $\theta$ maps all of the $c$-edges, the left edges and the right edges of $\G$ to edges of $\G$. As $a = 2$ it is also easy to verify that the $a$-edges of $\G$ are all mapped to edges of $\G$. To see that the $b$- and the $d$-edges are also mapped to edges of $\G$ observe that $m-2 \equiv m+2 \equiv 4-\delta \pmod{4}$. It is now not difficult to verify that the $b$- and the $d$-edges are mapped to edges of $\G$. For instance, in the case of $i \equiv 0 \pmod{4}$, we see that $B_i \theta = A_i$ is indeed adjacent to $C_{i+m-2}\theta = A_{i+m-2-m} = A_{i-2}$ and $C_{i+m+2}\theta = A_{i+m+2-m} = A_{i+2}$. We leave the remaining three cases to the reader. 
\end{proof}

It turns out that $\WH_8(2,1,0,5)$ and $\WH_8(2,1,4,5)$ are both vertex-transitive and are nonisomorphic. A symmetric drawing of each of them which is given in Figure~\ref{fig:sporadic} shows that the 6-fold rotation is an automorphism mapping an $A$-vertex to a $B$-vertex. Clearly, these two graphs belong to neither of the two families from Proposition~\ref{pro:VTfam1} and Proposition~\ref{pro:VTfam2}. 
Our goal in the remainder of this section is to prove that each vertex-transitive WH-graph which is not isomorphic to one of these two sporadic small graphs is isomorphic to a WH-graph from one of the two infinite families given by Proposition~\ref{pro:VTfam1} and Proposition~\ref{pro:VTfam2}.

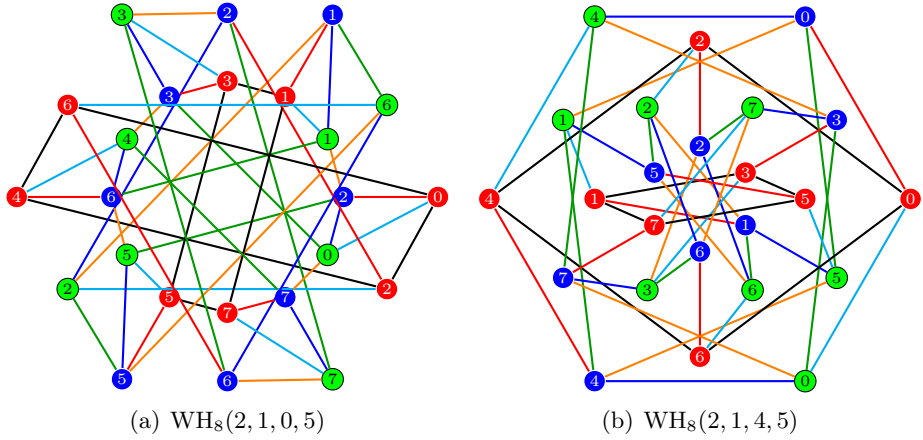
\begin{figure}[htbp]
\begin{center}
%%%%%%%%%%%%%%%
\subfigure[$\WH_{8}(2,1,0,5)$]{
\begin{tikzpicture}[scale= .7]
 \pgfmathsetmacro{\r}{0+60}
  \pgfmathsetmacro{\rr}{90-120}
   \pgfmathsetmacro{\rrr}{30+60}
   
  \pgfmathsetmacro{\s}{2.2}
    \pgfmathsetmacro{\ss}{3.5} 
       \pgfmathsetmacro{\sss}{2.2}  
 %%%Cycles[{{aa[0], bb[1], cc[3], aa[4], bb[5], cc[7]}, 
 %{aa[1], bb[3], bb[4], aa[5], bb[7], bb[0]}, 
 %{aa[2], cc[6], bb[2], aa[6], cc[2], bb[6]}, 
 %{aa[3], cc[4], cc[5], aa[7], cc[0], cc[1]}}],        

\node[vtx, white, fill = red] (A0) at (360*0/6  :4) {0};%{$A_{0}$};
\node[vtx, white, fill = blue] (B1) at (360*1/6  :4){1};% {$A_{4}$};
\node[vtx, , fill = green] (C3) at (360*2/6  :4) {3};%{$B_{0}$};
\node[vtx, white, fill = red] (A4) at (360*3/6 :4){4};% {$B_{4}$};
\node[vtx, white, fill = blue] (B5) at (360*4/6  :4){5};% {$C_{4}$};
\node[vtx, , fill = green] (C7) at (360*5/6 :4){7};% {$C_{0}$};

 %{aa[1], bb[3], bb[4], aa[5], bb[7], bb[0]}, 
\node[vtx, white, fill = red] (A1) at (360*0/6 +\r  :\s){1};% {$A_{1}$};
\node[vtx, white, fill = blue] (B3) at (360*1/6+\r :\s) {3};%{$C_{3}$};
\node[vtx, white, fill = blue] (B4) at (360*2/6 +\r :\s) {6};%{$C_{6}$};
\node[vtx, white, fill = red] (A5) at (360*3/6+\r :\s){5};% {$A_{5}$};
\node[vtx, white, fill = blue] (B7) at (360*4/6 +\r :\s) {7};%{$C_{7}$};
\node[vtx, white, fill = blue] (B0) at (360*5/6 +\r:\s) {2};%{$C_{2}$};

%{aa[2], cc[6], bb[2], aa[6], cc[2], bb[6]}, 
\node[vtx, white, fill = red] (A2) at (360*0/6+ \rr :\ss) {2};%{$A_{2}$};
\node[vtx, , fill = green] (C6) at (360*1/6  + \rr :\ss) {6};%{$C_{1}$};
\node[vtx, white, fill = blue] (B2) at (360*2/6 + \rr :\ss) {2};%{$B_{7}$};
\node[vtx, white, fill = red] (A6) at (360*3/6 + \rr:\ss) {6};%{$A_{6}$};
\node[vtx, , fill = green] (C2) at (360*4/6  + \rr :\ss) {2};%{$C_{5}$};
\node[vtx, white, fill = blue] (B6) at (360*5/6 + \rr :\ss) {6};%{$B_{3}$};

%{aa[3], cc[4], cc[5], aa[7], cc[0], cc[1]}
\node[vtx, white, fill = red] (A3) at (360*0/6 +\rrr :\sss) {3};%{$A_{3}$};
\node[vtx, , fill = green] (C4) at (360*1/6 +\rrr  :\sss) {4};%{$B_{2}$};
\node[vtx, , fill = green] (C5) at (360*2/6 +\rrr : \sss) {5};%{$B_{5}$};
\node[vtx, white, fill = red] (A7) at (360*3/6 +\rrr :\sss) {7};%{$A_{7}$};
\node[vtx,  , fill = green] (C0) at (360*4/6  +\rrr :\sss) {0};%{$B_{6}$};
\node[vtx, , fill = green] (C1) at (360*5/6 +\rrr :\sss) {1};%{$B_{1}$};

\foreach \j in {0,1,...,7}{
\draw[thick] let \n1 = {int(mod(\j+2, 8))} in (A\j) -- (A\n1);
\draw[thick,red] (A\j) -- (B\j);
\draw[thick, orange] let \n1 = {int(mod(\j+1, 8))} in (B\j) --  (C\n1); 
\draw[thick, blue] let \n1 = {int(mod(\j+0, 8))} in(B\j) -- (C\n1); 
\draw[thick, green!60!black] let \n1 = {int(mod(\j+5, 8))} in(B\j) -- (C\n1); 
\draw[thick,cyan] (A\j) -- (C\j);
}

%\node[ right = 2mm of C] {$\mathbb{Z}_{4}$};
\end{tikzpicture}
}
%%%%%%%%%%%%%%%%%%%%%%%%%%%
\subfigure[$\WH_{8}(2,1,4,5)$]{
\begin{tikzpicture}[scale= .7]
 \pgfmathtruncatemacro{\r}{180}
  \pgfmathtruncatemacro{\rr}{90}
   \pgfmathtruncatemacro{\rrr}{30}
   
  \pgfmathsetmacro{\s}{2}
    \pgfmathsetmacro{\ss}{3} 
       \pgfmathsetmacro{\sss}{1}    

\node[vtx, white, fill = red] (A0) at (360*0/6  :4) {0};%{$A_{0}$};
\node[vtx, white, fill = red] (A4) at (360*3/6  :4){4};% {$A_{4}$};
\node[vtx, white, fill = blue] (B0) at (360*1/6  :4) {0};%{$B_{0}$};
\node[vtx, white, fill = blue] (B4) at (360*4/6 :4){4};% {$B_{4}$};
\node[vtx, , fill = green] (C4) at (360*2/6  :4){4};% {$C_{4}$};
\node[vtx, , fill = green] (C0) at (360*5/6 :4){0};% {$C_{0}$};

\node[vtx, white, fill = red] (A1) at (360*0/6 +\r  :\s){1};% {$A_{1}$};
\node[vtx, , fill = green] (C3) at (360*1/6+\r :\s) {3};%{$C_{3}$};
\node[vtx, , fill = green] (C6) at (360*2/6 +\r :\s) {6};%{$C_{6}$};
\node[vtx, white, fill = red] (A5) at (360*3/6+\r :\s){5};% {$A_{5}$};
\node[vtx, , fill = green] (C7) at (360*4/6 +\r :\s) {7};%{$C_{7}$};
\node[vtx, , fill = green] (C2) at (360*5/6 +\r:\s) {2};%{$C_{2}$};

\node[vtx, white, fill = red] (A2) at (360*0/6+ \rr :\ss) {2};%{$A_{2}$};
\node[vtx, , fill = green] (C1) at (360*1/6  + \rr :\ss) {1};%{$C_{1}$};
\node[vtx, white, fill = blue] (B7) at (360*2/6 + \rr :\ss) {7};%{$B_{7}$};
\node[vtx, white, fill = red] (A6) at (360*3/6 + \rr:\ss) {6};%{$A_{6}$};
\node[vtx, , fill = green] (C5) at (360*4/6  + \rr :\ss) {5};%{$C_{5}$};
\node[vtx, white, fill = blue] (B3) at (360*5/6 + \rr :\ss) {3};%{$B_{3}$};

\node[vtx, white, fill = red] (A3) at (360*0/6 +\rrr :\sss) {3};%{$A_{3}$};
\node[vtx, white, fill = blue] (B2) at (360*1/6 +\rrr  :\sss) {2};%{$B_{2}$};
\node[vtx, white, fill = blue] (B5) at (360*2/6 +\rrr : \sss) {5};%{$B_{5}$};
\node[vtx, white, fill = red] (A7) at (360*3/6 +\rrr :\sss) {7};%{$A_{7}$};
\node[vtx, white , fill = blue] (B6) at (360*4/6  +\rrr :\sss) {6};%{$B_{6}$};
\node[vtx, white, fill = blue] (B1) at (360*5/6 +\rrr :\sss) {1};%{$B_{1}$};

\foreach \j in {0,1,...,7}{
\draw[thick] let \n1 = {int(mod(\j+2, 8))} in (A\j) -- (A\n1);
\draw[thick,red] (A\j) -- (B\j);
\draw[thick, orange] let \n1 = {int(mod(\j+1, 8))} in (B\j) --  (C\n1); 
\draw[thick, blue] let \n1 = {int(mod(\j+4, 8))} in(B\j) -- (C\n1); 
\draw[thick, green!60!black] let \n1 = {int(mod(\j+5, 8))} in(B\j) -- (C\n1); 
\draw[thick,cyan] (A\j) -- (C\j);
}

%\node[ right = 2mm of C] {$\mathbb{Z}_{4}$};
\end{tikzpicture}
}
\caption{The two sporadic vertex-transitive WH-graphs.}
\label{fig:sporadic}
\end{center}
\end{figure}

We start by introducing some further terminology and by making a few useful observations, which we record in Proposition~\ref{pro:VT_basic} for ease of reference.
Let $\G = \WH_n(a,b,c,d)$ be a vertex-transitive WH-graph and let $H = \la \rho, \tau \ra$, where $\rho$ and $\tau$ are as in Lemma~\ref{le:sym}. Since $\G$ is vertex-transitive, Theorem~\ref{the:ET} implies that the four edges incident to $A_0$ are not in the same $\Aut(\G)$-orbit. As these four edges come from two $H$-orbits, $\Aut(\G)$ has precisely two orbits on the edge set of $\G$ and each vertex is incident to two edges from each of these two orbits. Let $\cR$ be the $\Aut(\G)$-orbit of the edge $A_0A_a$. Applying $\rho^{-a}$ we see that $A_{-a}A_0 \in \cR$, and so $A_0A_a$ and $A_{-a}A_0$ are the two edges from $\cR$ incident to $A_0$. Letting $\cB$ be the $\Aut(\G)$-orbit of the edge $A_0B_0$ we thus see that $\cR$ and $\cB$ are the two $\Aut(\G)$-orbits on the edge set of $\G$. We call the edges from $\cR$ {\em red} and the ones from $\cB$ {\em blue}. Since each vertex is incident to two red and to two blue edges this gives rise to {\em red cycles} and {\em blue cycles}. 

The red cycle containing the red edge $A_0A_a$ is $(A_0,A_a,A_{2a},\ldots,A_{-a})$ and is thus of length $n/\gcd(a,n)$. Since $A_0B_0$ is blue, there is precisely one $x \in \{b,c,d\}$ such that $B_0C_x$ is also blue. With no loss of generality assume $B_0C_c$ is blue. Applying $\rho^c$ we thus see that the blue cycles are of the form $(A_i,B_i,C_{i+c},A_{i+c},B_{i+c},C_{i+2c},A_{i+2c},\ldots,C_{i-c})$, $i \in \ZZ_n$. It follows that the red cycle through $B_0$ is $(B_0,C_d,B_{d-b},C_{2d-b},B_{2(d-b)},\ldots , C_{b})$ and is thus of length $2n/\gcd(d-b,n)$.  Since all red cycles must be of the same length, we must have that $2\gcd(a,n) = \gcd(d-b,n)$. In particular, $n$ and the order $|a|$ of $a$ are both even. Our agreement regarding the colors of edges can be seen in Figure~\ref{fig:colors} where the choice of colors is represented via the Wooly Hat diagram and is shown in the case of the WH-graph $\WH_4(1,3,0,1)$.

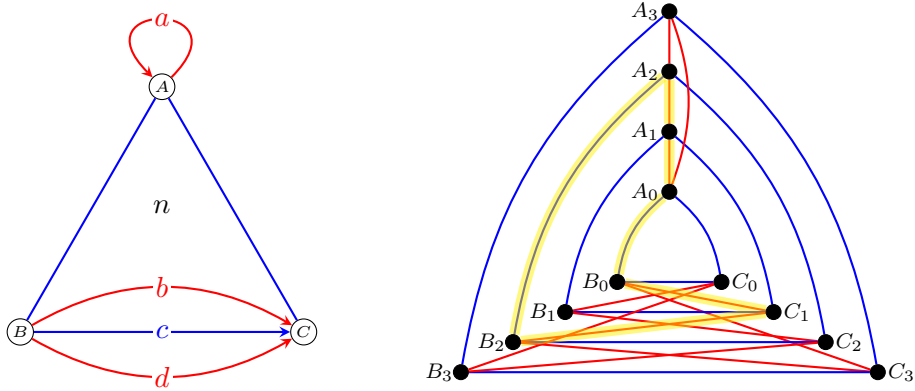
\begin{figure}[htbp]
\begin{center}
\subfigure
{
\begin{tikzpicture}[scale = 1.45]
\node[vtx] (A) at (90: 1.5) {$A$};
\node[vtx] (B) at (90+120: 1.5) {$B$};
\node[vtx] (C) at (90+240: 1.5) {$C$};
\draw[->, thick, red] (A) to[in = 90+45, out = 45,looseness = 3, distance = 1 cm] node[lbl]{$a$} %{$1$} 
(A);
\draw[thick, blue] (A) -- (B);
\draw[thick,->, red] (B) to[bend left = 30] node[lbl] {$b$}%{$0$} 
(C); 
\draw[thick, ->, blue] (B) to[bend left = 0] node[lbl] {$c$}%{$1$} 
(C); 
\draw[thick,->, red] (B) to[bend left = -30] node[lbl] {$d$}%{$3$} 
(C); 
\draw[thick, blue] (A) -- (C);
%\node[ right = 2mm of C] {$\mathbb{Z}_{n}$};
\node[ below = 12mm of A] {$n$};
\end{tikzpicture}
}
%%%%%%%%%%%%%%%%%%%
\hspace{8mm}
%\hfill
%%%%%%%%%%%%%%%%%%%
\subfigure{
\begin{tikzpicture}[scale = .8]
\foreach \j in {0,1,2,3}{
 \pgfmathtruncatemacro{\rad}{\j+1};
\node[vtx, fill=black, inner sep = 2pt, label={[label distance=-3pt]left:{\scriptsize $A_{\j}$}}] (A\j) at (360*0/4+90:\rad) {};%{$\j$};%{$A_{\j}$};
\node[vtx, fill = black, inner sep = 2pt, label={[label distance=-4pt]180:{\scriptsize $B_{\j}$}}]  (B\j) at (360*1/3+90: \rad){}; %{$B_{\j}$};
\node[vtx, fill = black, inner sep = 2pt, label={[label distance=-3pt]0:{\scriptsize $C_{\j}$}}]  (C\j) at (360*2/3+90: \rad) {};%{$C_{\j}$};
}
\begin{scope}[on background layer]
\foreach \j in {0,1,2,3}{
%\draw[thick] let \n1 = {int(mod(\j+1, 4))} in (A\j) -- (A\n1);
\draw[thick, blue] (A\j) to[bend right = 20] (B\j);
\draw[ thick, blue] let \n1 = {int(mod(\j+0, 4))} in (B\j) to[bend right = 0] (C\n1); 
\draw[thick, red] let \n1 = {int(mod(\j+1, 4))} in(B\j) -- (C\n1); 
\draw[ thick, red] let \n1 = {int(mod(\j+3, 4))} in(B\j) -- (C\n1); 
\draw[thick, blue] (A\j) to[bend left = 20](C\j);
}
\foreach \j in {0,1,2}{
\draw[ thick, red] let \n1 = {int(mod(\j+1, 4))} in (A\j) -- (A\n1);
}
\draw[ thick, red] (A3) to[bend left = 20] (A0);
\draw[opacity = .5, yellow, line width = 4 pt] (A0) -- (A1) -- (A2);
\draw[opacity = .5, yellow, line width = 4 pt] (B2) -- (C1) -- (B0);
\draw[opacity = .5, yellow, line width = 4 pt] (A0) to [bend right = 20] (B0);
\draw[opacity = .5, yellow, line width = 4 pt] (A2) to [bend right = 20] (B2);
\end{scope}
%\begin{scope}[on background layer]
%\draw[dotted, shorten >= -.5cm, shorten <= -.5cm] (A0.0) -- (A2.180);
%\end{scope}
%\node[ right = 2mm of C] {$\mathbb{Z}_{4}$};
\end{tikzpicture}
}
\caption{The red and blue edges and a basic $6$-cycle in $\WH_4(1,3,0,1)$.}
\label{fig:colors}
\end{center}
\end{figure}

Recall that each WH-graph possesses canonical $6$-cycles and note that each canonical $6$-cycle consists of four red and two blue edges where the blue edges are antipodal on this cycle. We call all $6$-cycles consisting of four red and two blue edges where the blue edges are antipodal on this cycle {\em basic}. Since $\G$ is vertex-transitive, there exists a basic $6$-cycle containing the red $2$-path $P = (A_0,A_a,A_{2a})$. Since no red $2$-path can connect a $B$-vertex to a $C$-vertex it follows that the two blue edges of any basic $6$-cycle containing $P$ are either both left edges or are both right edges. It is thus clear that one of $2a = d-b$ or $2a = b-d$ must hold. With no loss of generality we assume that $2a = d-b$ holds. An example of a basic $6$-cycle in the graph $\WH_4(1,3,0,1)$ is highlighted in Figure~\ref{fig:colors}. 

Suppose one of $b$ and $d$ is $0$, say $b = 0$. Then the red edge $B_0C_b$ lies on a $3$-cycle containing two blue edges. Since $\cR$ is an $\Aut(\G)$-orbit the same should hold for the red edge $B_0C_d$, clearly forcing $d = 0$ and thus contradicting the fact that $d \neq b$. Therefore, $b$ and $d$ must both be nonzero. 

We finally look at the intersections of blue and red cycles. The intersection of the blue and the red cycle containing the vertex $A_0$ is
\begin{equation}
\label{eq:int1}
	\{A_i \colon i \in \la a \ra \cap \la c \ra\}.
\end{equation}
On the other hand, taking into account that $d-b = 2a$, the intersection of the blue and the red cycle containing the vertex $B_0$ is 
$$
	\{B_i \colon i \in \la 2a \ra \cap \la c \ra\} \cup \{C_i \colon i \in (b + \la 2a \ra) \cap \la c \ra\}. 
$$
Note that the distance on the corresponding blue cycle between any two vertices from~\eqref{eq:int1} is a multiple of $3$. However, the distance on the corresponding blue cycle between $B_0$ and any potential vertex from $\{C_i \colon i \in (b + \la 2a \ra) \cap \la c \ra\}$ is never a multiple of $3$. By vertex-transitivity we thus get
\begin{equation}
\label{eq:int2}
	(b + \la 2a \ra) \cap \la c \ra = \emptyset\quad \text{and}\quad \la a \ra \cap \la c \ra = \la 2a \ra \cap \la c \ra.
\end{equation}
This shows that $\lcm(\gcd(a,n),\gcd(c,n)) = \lcm(\gcd(2a,n),\gcd(c,n))$. Since the order of $a$ is even, $\gcd(2a,n) = 2\gcd(a,n)$, and so $\gcd(c,n)$ is even. In fact, if we define the {\em $2$-part} of a positive integer $m$ to be $2^k$, where $k \geq 0$ is the largest integer such that $2^k$ divides $m$, we have thus shown that the $2$-part of $\gcd(c,n)$ is larger than the $2$-part of $\gcd(a,n)$. We record all of the above observations for ease of future reference. 

\begin{proposition}
\label{pro:VT_basic}
Let $\G = \WH_n(a,b,c,d)$ be a vertex-transitive WH-graph. Then, changing the roles of $b$, $c$ and $d$ if necessary, the following hold:
\begin{itemize}
\itemsep = 0pt
\item $n$ is even.
\item $b$ and $d$ are both nonzero.
\item $2a = d-b$.
\item The $2$-part of $\gcd(c,n)$ is larger than the $2$-part of $\gcd(a,n)$. In particular, $c$ is even.
\item The two $\Aut(\G)$-orbits on the edge set of $\G$ are $\cB$ and $\cR$, where $\cB$ consists of all the left, the right and the $c$-edges, while $\cR$ consists of all the $a$-, the $b$- and the $d$-edges.  
\end{itemize}
\end{proposition}

Our classification of the vertex-transitive WH-graphs relies heavily on the analysis of certain structures in these graphs in relation to the two colors of edges. In particular, the following notion will play an important role. We say that two vertices of a WH-graph are {\em basically $\cR$-antipodal} if there exists a basic $6$-cycle on which each of them is incident to two red edges (and so they are antipodal on this cycle). We also say that a path, a walk or a cycle is {\em red} ({\em blue}, respectively) if all of its edges are red (blue, respectively) and is {\em alternating} if no two consecutive edges of it are of the same color. 

Before proceeding with our analysis a remark is in order. Observe that the above result ensures that for a vertex-transitive $\G = \WH_n(a,b,c,d)$ the order of $a$ is even. Therefore, $d-b = 2a$ implies that the orbits of the subgroup $\la \rho^{2a}\ra$, where $\rho$ is as in~\eqref{eq:rho}, are blocks of imprimitivity for $\Aut(\G)$ (each of them coincides with a set consisting of every other vertex of a red cycle). One could thus consider the quotient graph with respect to these blocks (where we regard the ``double'' red edge as a single edge) and obtain a cubic graph on which the induced action of $\Aut(\G)$ is vertex-transitive. This graph is of course a tricirculant (in the sense that it admits a semiregular automorphism with three orbits), and so one could use the results of~\cite{PotTol20} where all cubic vertex-transitive tricirculants were classified. Nevertheless, the relevant result, namely~\cite[Theorem~4.2]{PotTol20}, only classifies the graphs up to isomorphism, its proof is very long and technical (about 11 pages), and even so one would still need to carefully examine what the conclusions of that result on our quotient graph say about the original parameters $n$, $a$, $b$, $c$ and $d$. Especially so since the quotient might (and in fact quite often is) be one of the sporadic small graphs of orders $6$ (the prism) and $12$ (the truncation of $K_4$) from~\cite[Table~1]{PotTol20}. This is why we have decided to take a direct approach without referring to the results of~\cite{PotTol20}. 

We start our analysis with a useful observation regarding red $2$-paths and basic $6$-cycles.

\begin{lemma}
\label{le:basic2R}
Let $\G = \WH_n(a,b,c,d)$ be a vertex-transitive WH-graph where the parameters $n$, $a$, $b$, $c$ and $d$ are as in Proposition~\ref{pro:VT_basic}. Then the following hold:
\begin{itemize}
\itemsep = 0pt
\item If $4a \neq 0$ then each red $2$-path lies on precisely two basic $6$-cycles.
\item If $4a = 0$ then each red $2$-path lies on precisely four basic $6$-cycles.
\end{itemize}
\end{lemma}

\begin{proof}
Let us consider the basic $6$-cycles containing the red $2$-path $(B_0,C_d,B_{d-b})$. Clearly, if at least one of its two blue edges is a left edge (and so there are also two $a$-edges on this cycle) then both are left edges. If this is indeed the case, then the $6$-cycle must contain $(A_0,B_0,C_d,B_{d-b},A_{d-b})$. Since $d-b = 2a$, one of the possibilities for the remaining vertex of this basic $6$-cycle is $A_a$. The only other possibility is $A_{-a}$ which gives rise to a $6$-cycle if and only if $-2a = 2a$ (which is equivalent to $4a = 0$). A similar situation occurs if both blue edges are $c$-edges, where we seek for red $2$-paths connecting $C_c$ and $C_{d-b+c}$. One possibility results in a canonical $6$-cycle while the other is only possible if $2(d-b) = 0$, which is of course equivalent to $4a = 0$.
\end{proof}

Let $\G = \WH_n(a,b,c,d)$ be a vertex-transitive WH-graph, where $a,b,c,d$ are as in Proposition~\ref{pro:VT_basic}. Consider the part of $\G$ depicted on Figure~\ref{fig:alt8_1}.
%\begin{figure}[!h]
%\begin{center}
%	\includegraphics[scale=0.7]{alt8_1_2.pdf}
%	\caption{The situation around the blue edge $A_aB_a$.}
%	\label{fig:alt8_1}
%\end{center}
%\end{figure}
\begin{figure}[htbp]
\begin{center}
\begin{tikzpicture}[line cap=round,line join=round,>=triangle 45,x=.9cm,y=.9cm, scale = 0.85]
\clip(1,3) rectangle (15,7);
\draw [line width=1pt,color=red] (3,5)-- (4,6);
\draw [line width=1pt,color=blue] (4,6)-- (6,6);
\draw [line width=1pt,color=red] (6,6)-- (7,5);
\draw [line width=1pt,color=red] (7,5)-- (6,4);
\draw [line width=1pt,color=blue] (6,4)-- (4,4);
\draw [line width=1pt,color=red] (4,4)-- (3,5);
\draw [line width=1pt,color=blue] (7,5)-- (9,5);
\draw [line width=1pt,color=red] (9,5)-- (10,6);
\draw [line width=1pt,color=blue] (10,6)-- (12,6);
\draw [line width=1pt,color=red] (12,6)-- (13,5);
\draw [line width=1pt,color=red] (13,5)-- (12,4);
\draw [line width=1pt,color=blue] (12,4)-- (10,4);
\draw [line width=1pt,color=red] (10,4)-- (9,5);
\draw [color=blue,shift={(8,-12)},line width=1pt]  plot[domain=1.2847448850775784:1.856847768512215,variable=\t]({1*17.720045146669353*cos(\t r)+0*17.720045146669353*sin(\t r)},{0*17.720045146669353*cos(\t r)+1*17.720045146669353*sin(\t r)});
\draw (3.570031197790194,6.7) node[anchor=north west] {{\small $B_0$}};
\draw (5.497809785201585,6.7) node[anchor=north west] {{\small $A_0$}};
\draw (9.238275701074432,6.7) node[anchor=north west] {{\small $C_{a+b}$}};
\draw (11.482555250598143,6.7) node[anchor=north west] {{\small $A_{a+b}$}};

\draw (6.70626800417589,4.8) node[anchor=north west] {{\small $A_a$}};
\draw (8.5,4.8) node[anchor=north west] {{\small $B_a$}};

\draw (9.612322292661718,3.8) node[anchor=north west] {{\small $C_{3a+b}$}};
\draw (11.396236806385692,3.8) node[anchor=north west] {{\small $A_{3a+b}$}};
\draw (5.612901044151519,3.8) node[anchor=north west] {{\small $A_{2a}$}};
\draw (3.45493993884026,3.8) node[anchor=north west] {{\small $B_{2a}$}};
\draw (13,5) node[anchor=north west] {{\small $A_{2a+b}$}};
\draw (1.7,5) node[anchor=north west] {{\small $C_{2a+b}$}};
\begin{scriptsize}
\draw [fill=black] (4,4) circle (2.5pt);
\draw [fill=black] (6,4) circle (2.5pt);
\draw [fill=black] (7,5) circle (2.5pt);
\draw [fill=black] (3,5) circle (2.5pt);
\draw [fill=black] (4,6) circle (2.5pt);
\draw [fill=black] (6,6) circle (2.5pt);
\draw [fill=black] (9,5) circle (2.5pt);
\draw [fill=black] (10,6) circle (2.5pt);
\draw [fill=black] (12,6) circle (2.5pt);
\draw [fill=black] (10,4) circle (2.5pt);
\draw [fill=black] (12,4) circle (2.5pt);
\draw [fill=black] (13,5) circle (2.5pt);
\end{scriptsize}
\end{tikzpicture}
\caption{The situation around the blue edge $A_aB_a$.}
	\label{fig:alt8_1}

\label{default}
\end{center}
\end{figure}
We see that the blue edge $A_aB_a$ has the property that each of its endvertices has a basically $\cR$-antipodal vertex, say $u$ for one endvertex and $v$ for the other, such that $u$ and $v$ are connected by a blue edge. Let $\rho$ and $\tau$ be as in Lemma~\ref{le:sym}. Since $\cB$ is an $\Aut(\G)$-orbit and $\tau\rho^c$ flips the blue edge $B_0C_c$, there is an automorphism $\eta \in \Aut(\G)$ such that $A_a\eta = B_0$ and $B_a\eta = C_c$. Then $C_{2a+b}\eta$ is basically $\cR$-antipodal to $B_0$, $A_{2a+b}\eta$ is basically $\cR$-antipodal to $C_c$, and $C_{2a+b}\eta$ is connected to $A_{2a+b}\eta$ by a blue edge. The next result gives the three possible conditions on the parameters that are obtained from the corresponding analysis of the different possibilities.  

\begin{proposition}
\label{pro:VTcond}
Let $\G = \WH_n(a,b,c,d)$ be a vertex-transitive WH-graph where the parameters $n$, $a$, $b$, $c$ and $d$ are as in Proposition~\ref{pro:VT_basic}. Then one of the following holds:
\begin{itemize}
\itemsep = 0pt
\item $2c = 3a + 3b$.
\item $4a = 0$ and $4c = 4b + n/2$.
\item $4c = 4a + 4b$. 
\end{itemize}
\end{proposition} 

\begin{proof}
Let $\eta \in \Aut(\G)$ be as in the paragraph preceding this proposition. We consider the possibilities for $C_{2a+b}\eta$ and $A_{2a+b}\eta$ and analyse the different corresponding conditions under which these two vertices are connected by a blue edge.

Suppose first that $4a \neq 0$. Then Lemma~\ref{le:basic2R} implies that there are only two possibilities for each of $C_{2a+b}\eta$ and $A_{2a+b}\eta$. It is easy to verify (see Figure~\ref{fig:alt8_2}) that in this case 
$$
	C_{2a+b}\eta \in \{A_{a+b}, C_{2a+2b-c}\} \quad \text{and}\quad A_{2a+b}\eta \in \{A_{c-a-b}, B_{2c-2a-2b}\}.
$$  
%\begin{figure}[!h]
%\begin{center}
%	\includegraphics[scale=0.7]{alt8_2.pdf}
%	\caption{The situation around the blue edge $B_0C_c$ when $4a \neq 0$.}
%	\label{fig:alt8_2}
%\end{center}
%\end{figure}
\begin{figure}[htbp]
\begin{center}
\begin{tikzpicture}[line cap=round,line join=round,>=triangle 45,x=1.5cm,y=1.5cm, scale = 0.9]
\clip(3,3.5) rectangle (12,6.5);
\draw [line width=1pt,color=blue] (5,6)-- (6,5.5);
\draw [line width=1pt,color=red] (6,5.5)-- (7,5);
\draw [line width=1pt,color=red] (7,5)-- (6,4.5);
\draw [line width=1pt,color=blue] (6,4.5)-- (5,4);
\draw [line width=1pt,color=red] (5,4)-- (4.3,4.3);
\draw [line width=1pt,color=blue] (7,5)-- (8,5);
\draw [line width=1pt,color=red] (8,5)-- (9,5.5);
\draw [line width=1pt,color=blue] (9,5.5)-- (10,6);
\draw [line width=1pt,color=red] (10.7,4.3)-- (10,4);
\draw [line width=1pt,color=blue] (10,4)-- (9,4.5);
\draw [line width=1pt,color=red] (9,4.5)-- (8,5);
\draw [line width=1pt,color=blue] (6,4.5)-- (5,5.4);
\draw [line width=1pt,color=blue] (6,5.5)-- (5,4.6);
\draw [line width=1pt,color=red] (4.3,4.3)-- (5,4.6);
\draw [line width=1pt,color=red] (4.3,5.7)-- (5,6);
\draw [line width=1pt,color=red] (4.3,5.7)-- (5,5.4);
\draw (4.8,6.4) node[anchor=north west] {{\small $A_b$}};
\draw (9.7,6.4) node[anchor=north west] {{\small $A_{c-b}$}};
\draw (10.6,6.1) node[anchor=north west] {{\small $A_{c-a-b}$}};
\draw (3.8328869372288734,6.1) node[anchor=north west] {{\small $A_{a+b}$}};
\draw (5.9,5.9) node[anchor=north west] {{\small $C_b$}};
\draw (8.5,5.9) node[anchor=north west] {{\small $B_{c-b}$}};
\draw (4.2,5.45) node[anchor=north west] {{\small $A_{2a+b}$}};
\draw (10,5.48) node[anchor=north west] {{\small $A_{c-2a-b}$}};
\draw (6.8,4.9) node[anchor=north west] {{\small $B_0$}};
\draw (7.8,4.9) node[anchor=north west] {{\small $C_c$}};
\draw (10,4.98) node[anchor=north west] {{\small $C_{2c-b}$}};
\draw (4.3,5.0) node[anchor=north west] {{\small $B_{b-c}$}};

\draw (3.6,4.2) node[anchor=north west] {{\small $C_{2a+2b-c}$}};
\draw (10.5,4.2) node[anchor=north west] {{\small $B_{2c-2a-2b}$}};

\draw (5.9,4.4) node[anchor=north west] {{\small $C_{2a+b}$}};
\draw (8.5,4.4) node[anchor=north west] {{\small $B_{c-2a-b}$}};
\draw (9.7,3.9) node[anchor=north west] {{\small $C_{2c-2a-b}$}};
\draw (4.8,3.9) node[anchor=north west] {{\small $B_{2a+b-c}$}};
\draw [line width=1pt,color=blue] (9,5.5)-- (9.994071718637914,4.600361113281613);
\draw [line width=1pt,color=blue] (9,4.5)-- (10.01024879460394,5.393037835616917);
\draw [line width=1pt,color=red] (10,6)-- (10.7,5.7);
\draw [line width=1pt,color=red] (10.7,5.7)-- (10.01024879460394,5.393037835616917);
\draw [line width=1pt,color=red] (9.994071718637914,4.600361113281613)-- (10.7,4.3);
\begin{scriptsize}
\draw [fill=black] (5,4) circle (2.5pt);
\draw [fill=black] (6,4.5) circle (2.5pt);
\draw [fill=black] (7,5) circle (2.5pt);
\draw [fill=black] (4.3,4.3) circle (2.5pt);
\draw [fill=black] (5,6) circle (2.5pt);
\draw [fill=black] (6,5.5) circle (2.5pt);
\draw [fill=black] (8,5) circle (2.5pt);
\draw [fill=black] (9,5.5) circle (2.5pt);
\draw [fill=black] (10,6) circle (2.5pt);
\draw [fill=black] (9,4.5) circle (2.5pt);
\draw [fill=black] (10,4) circle (2.5pt);
\draw [fill=black] (10.7,4.3) circle (2.5pt);
\draw [fill=black] (5,5.4) circle (2.5pt);
\draw [fill=black] (5,4.6) circle (2.5pt);
\draw [fill=black] (4.3,5.7) circle (2.5pt);
\draw [fill=black] (10.01024879460394,5.393037835616917) circle (2.5pt);
\draw [fill=black] (9.994071718637914,4.600361113281613) circle (2.5pt);
\draw [fill=black] (10.7,5.7) circle (2.5pt);
\end{scriptsize}
\end{tikzpicture}
\caption{The situation around the blue edge $B_0C_c$ when $4a \neq 0$.}
	\label{fig:alt8_2}
\end{center}
\end{figure}
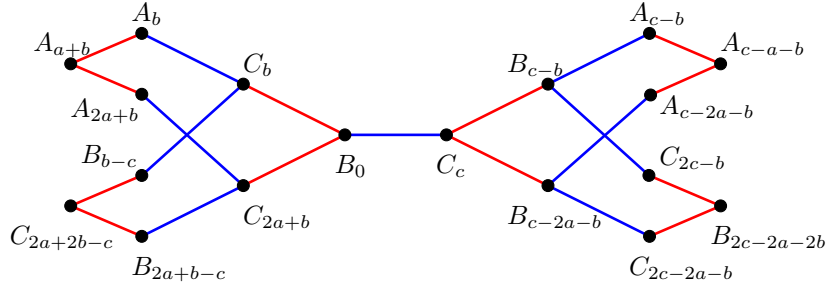
If $C_{2a+b}\eta = A_{a+b}$, then for this vertex to be connected to $A_{2a+b}\eta$ by a blue edge we must have that $A_{2a+b}\eta = B_{2c-2a-2b}$ and $a+b = 2c-2a-2b$, that is $2c = 3a+3b$. Similarly, if $A_{2a+b}\eta = A_{c-a-b}$, then $C_{2a+b}\eta = C_{2a+2b-c}$ and $c-a-b = 2a+2b-c$, again yielding $2c = 3a+3b$. The only remaining possibility is that $C_{2a+b}\eta = C_{2a+2b-c}$ and $A_{2a+b}\eta = B_{2c-2a-2b}$ with $3c-2a-2b = 2a+2b-c$, that is $4c = 4a+4b$.

Suppose now that $4a = 0$. Since $2a \neq 0$, this implies that $n$ is divisible by $4$ and that $a \in \{n/4, 3n/4\}$. This time Lemma~\ref{le:basic2R} implies that there are four possibilities for each of $C_{2a+b}\eta$ and $A_{2a+b}\eta$. However, as is evident from Figure~\ref{fig:alt8_3} the extra possibilities are just the antipodal vertices on the corresponding red $4$-cycles (which are just the vertices obtained by adding $n/2$ to the subscript) of the vertices obtained in the case of $4a \neq 0$. In other words, 
$$
	C_{2a+b}\eta \in \{A_{b+n/4}, A_{b+3n/4}, C_{2b-c}, C_{2b-c+n/2}\}
$$
and
$$
A_{2a+b}\eta \in \{A_{c-b+n/4}, A_{c-b+3n/4}, B_{2c-2b}, B_{2c-2b+n/2}\}.
$$
%\begin{figure}[!h]
%\begin{center}
%	\includegraphics[scale=0.7]{alt8_3.pdf}
%	\caption{The situation around the blue edge $B_0C_c$ when $4a = 0$.}
%	\label{fig:alt8_3}
%\end{center}
%\end{figure}
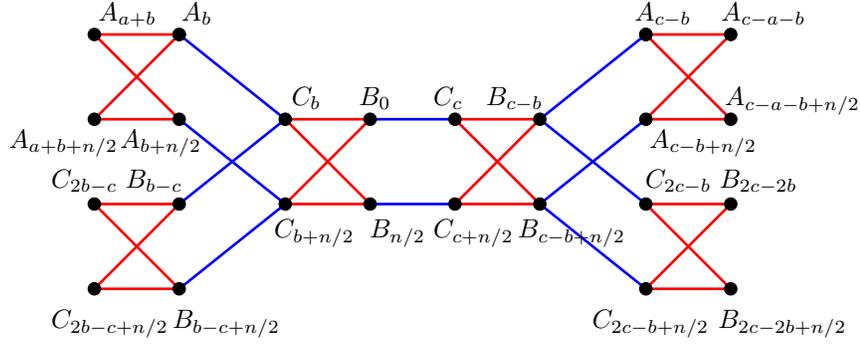
\begin{figure}[htbp]
\begin{center}
\begin{tikzpicture}[line cap=round,line join=round,>=triangle 45,x=1.5cm,y=1.5cm, scale = 0.94]
\clip(3.3,3.5) rectangle (11.8,7);
\draw [line width=1pt,color=blue] (5,6.4)-- (6,5.6);
\draw [line width=1pt,color=red] (6,5.6)-- (6.8,5.6);
\draw [line width=1pt,color=red] (6.8,5.6)-- (6,4.8);
\draw [line width=1pt,color=blue] (6,4.8)-- (5,4);
\draw [line width=1pt,color=red] (5,4)-- (4.2,4.8);
\draw [line width=1pt,color=blue] (6.8,5.6)-- (7.6,5.6);
\draw [line width=1pt,color=red] (7.6,5.6)-- (8.4,5.6);
\draw [line width=1pt,color=blue] (8.4,5.6)-- (9.4,6.4);
\draw [line width=1pt,color=red] (10.2,4.8)-- (9.4,4);
\draw [line width=1pt,color=blue] (9.4,4)-- (8.4,4.8);
\draw [line width=1pt,color=red] (8.4,4.8)-- (7.6,5.6);
\draw [line width=1pt,color=blue] (6,4.8)-- (5,5.6);
\draw [line width=1pt,color=blue] (6,5.6)-- (5,4.8);
\draw [line width=1pt,color=red] (4.2,4.8)-- (5,4.8);
\draw [line width=1pt,color=red] (4.2,6.4)-- (5,6.4);
\draw [line width=1pt,color=red] (4.2,6.4)-- (5,5.6);
\draw [line width=1pt,color=blue] (8.4,5.6)-- (9.4,4.8);
\draw [line width=1pt,color=blue] (8.4,4.8)-- (9.4,5.6);
\draw [line width=1pt,color=red] (9.4,6.4)-- (10.2,6.4);
\draw [line width=1pt,color=red] (10.2,6.4)-- (9.4,5.6);
\draw [line width=1pt,color=red] (9.4,4.8)-- (10.2,4.8);
\draw [line width=1pt,color=red] (4.2,4)-- (5,4);
\draw [line width=1pt,color=red] (5,4.8)-- (4.2,4);
\draw [line width=1pt,color=red] (5,6.4)-- (4.2,5.6);
\draw [line width=1pt,color=red] (6,4.8)-- (6.8,4.8);
\draw [line width=1pt,color=red] (4.2,5.6)-- (5,5.6);
\draw [line width=1pt,color=blue] (6.8,4.8)-- (7.6,4.8);
\draw [line width=1pt,color=red] (9.4,4)-- (10.2,4);
\draw [line width=1pt,color=red] (6,5.6)-- (6.8,4.8);
\draw [line width=1pt,color=red] (8.4,5.6)-- (7.6,4.8);
\draw [line width=1pt,color=red] (7.6,4.8)-- (8.4,4.8);
\draw [line width=1pt,color=red] (9.4,5.6)-- (10.2,5.6);
\draw [line width=1pt,color=red] (10.2,5.6)-- (9.4,6.4);
\draw [line width=1pt,color=red] (9.4,4.8)-- (10.2,4);

\draw (4.142151148686999,6.8) node[anchor=north west] {{\small $A_{a+b}$}};
\draw (4.927126064858327,6.8) node[anchor=north west] {{\small $A_b$}};
\draw (9.2,6.8) node[anchor=north west] {{\small $A_{c-b}$}};
\draw (10.015445610754615,6.8) node[anchor=north west] {{\small $A_{c-a-b}$}};
\draw (5.964414346941868,6) node[anchor=north west] {{\small $C_b$}};
\draw (6.6,6) node[anchor=north west] {{\small $B_0$}};
\draw (7.3,6) node[anchor=north west] {{\small $C_c$}};
\draw (7.8,6) node[anchor=north west] {{\small $B_{c-b}$}};
\draw (3.3,5.6) node[anchor=north west] {{\small $A_{a+b+n/2}$}};
\draw (4.35,5.6) node[anchor=north west] {{\small $A_{b+n/2}$}};
\draw (9.34,5.6) node[anchor=north west] {{\small $A_{c-b+n/2}$}};
\draw (10.05,6) node[anchor=north west] {{\small $A_{c-a-b+n/2}$}};
\draw (10,5.22) node[anchor=north west] {{\small $B_{2c-2b}$}};
\draw (9.300557740670012,5.22) node[anchor=north west] {{\small $C_{2c-b}$}};
\draw (4.4,5.22) node[anchor=north west] {{\small $B_{b-c}$}};
\draw (3.7,5.22) node[anchor=north west] {{\small $C_{2b-c}$}};
\draw (5.8,4.75) node[anchor=north west] {{\small $C_{b+n/2}$}};
\draw (6.7,4.75) node[anchor=north west] {{\small $B_{n/2}$}};
\draw (7.3,4.75) node[anchor=north west] {{\small $C_{c+n/2}$}};
\draw (8.1,4.75) node[anchor=north west] {{\small $B_{c-b+n/2}$}};
\draw (4.843021609554256,3.9) node[anchor=north west] {{\small $B_{b-c+n/2}$}};
\draw (9.987410792319924,3.9) node[anchor=north west] {{\small $B_{2c-2b+n/2}$}};
\draw (3.7,3.9) node[anchor=north west] {{\small $C_{2b-c+n/2}$}};
\draw (8.8,3.9) node[anchor=north west] {{\small $C_{2c-b+n/2}$}};
\begin{scriptsize}
\draw [fill=black] (5,4) circle (2.5pt);
\draw [fill=black] (5,6.4) circle (2.5pt);
\draw [fill=black] (5,5.6) circle (2.5pt);
\draw [fill=black] (5,4.8) circle (2.5pt);
\draw [fill=black] (5,6.4) circle (2.5pt);
\draw [fill=black] (6,4.8) circle (2.5pt);
\draw [fill=black] (6.8,5.6) circle (2.5pt);
\draw [fill=black] (6,5.6) circle (2.5pt);
\draw [fill=black] (6.8,4.8) circle (2.5pt);
\draw [fill=black] (4.2,5.6) circle (2.5pt);
\draw [fill=black] (4.2,4.8) circle (2.5pt);
\draw [fill=black] (4.2,6.4) circle (2.5pt);
\draw [fill=black] (4.2,4) circle (2.5pt);
\draw [fill=black] (7.6,5.6) circle (2.5pt);
\draw [fill=black] (7.6,4.8) circle (2.5pt);
\draw [fill=black] (8.4,4.8) circle (2.5pt);
\draw [fill=black] (8.4,5.6) circle (2.5pt);
\draw [fill=black] (9.4,6.4) circle (2.5pt);
\draw [fill=black] (9.4,4) circle (2.5pt);
\draw [fill=black] (9.4,5.6) circle (2.5pt);
\draw [fill=black] (9.4,4.8) circle (2.5pt);
\draw [fill=black] (9.4,6.4) circle (2.5pt);
\draw [fill=black] (10.2,6.4) circle (2.5pt);
\draw [fill=black] (10.2,4.8) circle (2.5pt);
\draw [fill=black] (10.2,4) circle (2.5pt);
\draw [fill=black] (10.2,5.6) circle (2.5pt);
\end{scriptsize}
\end{tikzpicture}
\caption{The situation around the blue edge $B_0C_c$ when $4a = 0$.}
	\label{fig:alt8_3}
\end{center}
\end{figure}
If one of $A_{b+n/4}, A_{b+3n/4}$ is connected by a blue edge to one of $B_{2c-2b}, B_{2c-2b+n/2}$, then $2c = 3b + n/4$ or $2c = 3b + 3n/4$. Exchanging the roles of $a$ and $-a$ if necessary (note that $2a = 2(-a)$, and so we do not violate the assumption $d = 2a+b$ this way) we thus see that $2c = 3a + 3b$. A similar conclusion can be made if one of $C_{2b-c}, C_{2b-c+n/2}$ is connected by a blue edge to one of $A_{c-b+n/4}, A_{c-b+3n/4}$. Finally, suppose one of $C_{2b-c}, C_{2b-c+n/2}$ is connected by a blue edge to one of $B_{2c-2b}, B_{2c-2b+n/2}$. Then one of $4c = 4b (= 4b + 4a)$ or $4c = 4b + n/2$ holds, as claimed.
\end{proof}

We now show that each of the three possibilities from the above proposition leads to the graphs from Proposition~\ref{pro:VTfam1}, Proposition~\ref{pro:VTfam2} or to the two sporadic examples with $n = 8$, mentioned at the beginning of this section. The first possibility is easy to handle, while the other two require a more detailed analysis of the structure of the corresponding graphs.

\begin{proposition}
\label{pro:VTcond1}
Let $\G = \WH_n(a,b,c,d)$ be a vertex-transitive WH-graph where the parameters $n$, $a$, $b$, $c$ and $d$ are as in Proposition~\ref{pro:VT_basic}. If $2c = 3a+3b$ then $a,b$ and $d$ are all odd, and so $\G$ belongs to the family of graphs from Proposition~\ref{pro:VTfam1}.
\end{proposition}

\begin{proof}
Proposition~\ref{pro:VT_basic} implies that $n$ is even, and so $2c = 3a+3b$ implies that $a+b$ is also even. Therefore, $a$ and $b$ are of the same parity. Since $d = 2a+b$ and $c$ is even, the fact that by definition $\G$ is connected implies that $a$, $b$ and $d$ must all be odd.
\end{proof}

\begin{proposition}
\label{pro:VTcond2}
Let $\WH_n(a,b,c,d)$ be a vertex-transitive WH-graph where the parameters $n$, $a$, $b$, $c$ and $d$ are as in Proposition~\ref{pro:VT_basic}, $4a = 0$ and $4c = 4b+n/2$. Then $n = 8$, $c \in \{0,4\}$ and, exchanging the pair $\{b,d\}$ with $\{-b,-d\}$ if necessary, $\{b,d\} = \{1,5\}$.
\end{proposition}

\begin{proof}
Let $\G = \WH_n(a,b,c,d)$ and suppose $4a = 0$ and $4c = 4b + n/2$. Since $2a \neq 0$ but $4a = 0$, $n$ is divisible by $4$ and $a \in \{n/4, 3n/4\}$. Consequently, as $4(c-b) = n/2$, $n$ in fact must be divisible by $8$. Since $a$ is thus even, $b$ must be odd (otherwise, as $c$ is even and $d = b + 2a$, $\G$ is not connected). Moreover, no odd prime divisor of $n$ can divide $b$ as otherwise it would also divide $c$ (and of course $a$), again contradicting the assumption that $\G$ is connected. Therefore, $b$ is coprime to $n$, and so Lemma~\ref{le:iso} implies that we can assume $b = 1$. We thus get $4c = 4 + n/2$, and so $2(c-1) \in \{n/4, 3n/4\} = \{a,-a\}$. Consider now the situation around the blue $2$-path $(A_0,B_0,C_c)$ depicted on Figure~\ref{fig:7cyc}.
%\begin{figure}[!h]
%\begin{center}
%	\includegraphics[scale=0.8]{cyc7.pdf}
%	\caption{The situation around the blue $2$-path $(A_0,B_0,C_c)$.}
%	\label{fig:7cyc}
%\end{center}
%\end{figure}
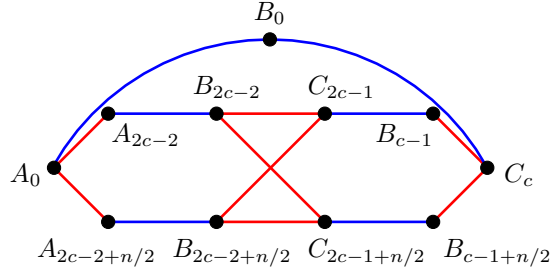
\begin{figure}[htbp]
\begin{center}
\definecolor{ffqqtt}{rgb}{1,0,0.2}
\definecolor{ccqqqq}{rgb}{0.8,0,0}
\definecolor{qqqqff}{rgb}{0,0,1}
\begin{tikzpicture}[line cap=round,line join=round,>=triangle 45,x=1.8cm,y=1.8cm]
\clip(4.3,4.2) rectangle (8.5,6.5);
\draw [line width=1pt,color=blue] (5.2,5.6)-- (6,5.6);
\draw [line width=1pt,color=red] (6,5.6)-- (6.8,5.6);
\draw [line width=1pt,color=red] (6.8,5.6)-- (6,4.8);
\draw [line width=1pt,color=blue] (6,4.8)-- (5.2,4.8);
\draw [line width=1pt,color=blue] (6.8,5.6)-- (7.6,5.6);
\draw [line width=1pt,color=red] (7.6,5.6)-- (8,5.2);
\draw [line width=1pt,color=red] (4.8,5.2)-- (5.2,4.8);
\draw [line width=1pt,color=red] (6,4.8)-- (6.8,4.8);
\draw [line width=1pt,color=red] (4.8,5.2)-- (5.2,5.6);
\draw [line width=1pt,color=blue] (6.8,4.8)-- (7.6,4.8);
\draw [line width=1pt,color=red] (6,5.6)-- (6.8,4.8);
\draw [line width=1pt,color=red] (8,5.2)-- (7.6,4.8);

\draw (4.4,5.3) node[anchor=north west] {{\small $A_0$}};
\draw (8.05,5.3) node[anchor=north west] {{\small $C_c$}};

\draw (5.75,5.95) node[anchor=north west] {{\small $B_{2c-2}$}};
\draw (6.6,5.95) node[anchor=north west] {{\small $C_{2c-1}$}};

\draw (7.1,5.6) node[anchor=north west] {{\small $B_{c-1}$}};
\draw (5.15,5.6) node[anchor=north west] {{\small $A_{2c-2}$}};

\draw (4.6,4.75) node[anchor=north west] {{\small $A_{2c-2+n/2}$}};
\draw (5.6,4.75) node[anchor=north west] {{\small $B_{2c-2+n/2}$}};
\draw (6.6,4.75) node[anchor=north west] {{\small $C_{2c-1+n/2}$}};
\draw (7.6,4.75) node[anchor=north west] {{\small $B_{c-1+n/2}$}};

\draw [shift={(6.393526435268658,4.314346988679726)},line width=1pt,color=blue]  plot[domain=0.5038427176587198:2.634321270934659,variable=\t]({1*1.8344314025444763*cos(\t r)+0*1.8344314025444763*sin(\t r)},{0*1.8344314025444763*cos(\t r)+1*1.8344314025444763*sin(\t r)});
\draw (6.2,6.5) node[anchor=north west] {{\small $B_0$}};

\begin{scriptsize}
\draw [fill=black] (5.2,5.6) circle (2.5pt);
\draw [fill=black] (5.2,4.8) circle (2.5pt);
\draw [fill=black] (6,4.8) circle (2.5pt);
\draw [fill=black] (6.8,5.6) circle (2.5pt);
\draw [fill=black] (6,5.6) circle (2.5pt);
\draw [fill=black] (6.8,4.8) circle (2.5pt);
\draw [fill=black] (4.8,5.2) circle (2.5pt);
\draw [fill=black] (7.6,5.6) circle (2.5pt);
\draw [fill=black] (7.6,4.8) circle (2.5pt);
\draw [fill=black] (8,5.2) circle (2.5pt);

\draw [fill=black] (6.39448138340117,6.148778142665986) circle (2.5pt);
\end{scriptsize}
\end{tikzpicture}
\caption{The situation around the blue $2$-path $(A_0,B_0,C_c)$.}
	\label{fig:7cyc}
\end{center}
\end{figure}
By vertex-transitivity and since $\tau \in \Aut(\G)$ there exists $\eta \in \Aut(\G)$ mapping $A_0$ to $B_0$, $B_0$ to $A_0$ and $C_c$ to $C_0$. Note that $\{B_{2c-2}\eta, B_{2c-2+n/2}\eta\}$ and $\{C_{2c-1}\eta, C_{2c-1+n/2}\eta\}$ are pairs of antipodal vertices on the same red $4$-cycle. Since $\{A_{2c-2}\eta, A_{2c-2+n/2}\eta\} = \{C_1, C_{1+n/2}\}$ and $\{B_{c-1}\eta, B_{c-1+n/2}\eta\} = \{B_{-1}, B_{-1+n/2}\}$, we see that
$$
	\{B_{2c-2}\eta, B_{2c-2+n/2}\eta\} \in \{\{A_1, A_{1+n/2}\}, \{B_{1-c}, B_{1-c+n/2}\}\}
$$
and
$$
	\{C_{2c-1}\eta, C_{2c-1+n/2}\eta\} \in \{\{A_{-1}, A_{-1+n/2}\}, \{C_{c-1}, C_{c-1+n/2}\}\}.
$$
Since all of the indices of these two $B$-vertices and these two $C$-vertices are odd (recall that $c$ and $n/2$ are both even) and $b = 1$ and $d = 1 + n/2$ are also odd, it thus follows that there are no red edges between the pairs $\{B_{1-c}, B_{1-c+n/2}\}$ and $\{C_{c-1}, C_{c-1+n/2}\}$, and so the only possibility is that $\{A_1, A_{1+n/2}\}$ and $\{A_{-1}, A_{-1+n/2}\}$ are the two pairs of antipodal vertices on the same red $4$-cycle. But then $-1+n/4 \in \{1,1+n/2\}$ which is clearly only possible if $-1+n/4 = 1$, implying that $n = 8$. To complete the proof, observe that since $a$ is even, Proposition~\ref{pro:VT_basic} implies that $c$ is divisible by $4$, and so $c \in \{0,4\}$, completing the proof.
\end{proof}

We now finally analyse the last of the three possibilities from Proposition~\ref{pro:VTcond}. In fact, in view of Proposition~\ref{pro:VTcond1}, it suffices to consider the possibility that $4c = 4a + 4b$ but $2c \neq 3a + 3b$. Before making this analysis we introduce some more terminology. Suppose $\G = \WH_n(a,b,c,d)$ is a vertex-transitive WH-graph where the parameters $n$, $a$, $b$, $c$ and $d$ are as in Proposition~\ref{pro:VT_basic} and $4c = 4a+4b$ but $2c \neq 3a + 3b$. Note that this implies that we cannot at the same time have $4a = 0$ with $4c = 4b + n/2$. 

Suppose first that $4a \neq 0$. The proof of Proposition~\ref{pro:VTcond} then reveals that for the blue edge $B_0C_c$ there is a unique basically $\cR$-antipodal vertex of $B_0$ (namely $C_{2a+2b-c}$) and a unique basically $\cR$-antipodal vertex of $C_c$ (namely $B_{2c-2a-2b}$) such that these two vertices are connected by a blue edge. Since $\cB$ is an $\Aut(\G)$-orbit this means that for any blue edge $uv$ there is a unique blue edge $u'v'$ such that $u'$ is basically $\cR$-antipodal to $v$ and $v'$ is basically $\cR$-antipodal to $u$. In this case we say that the blue edge $u'v'$ {\em corresponds} to the blue edge $uv$ and more precisely that the blue arc $(u',v')$ {\em corresponds} to the blue arc $(u,v)$. Of course, the relation of correspondence is symmetric and if $(u',v')$ corresponds to $(u,v)$, then $(v',u')$ corresponds to $(v,u)$. Using the automorphism $\rho$ from~\eqref{eq:rho} and the fact that we have the situation from Figure~\ref{fig:alt8_1} we clearly get the following. 

\begin{lemma}
\label{le:aux1}
Let $\WH_n(a,b,c,d)$ be a vertex-transitive WH-graph where the parameters $n$, $a$, $b$, $c$ and $d$ are as in Proposition~\ref{pro:VT_basic}. Suppose $4a \neq 0$, $4c = 4a+4b$ and $2c \neq 3a+3b$. Then the following holds for each $i \in \ZZ_n$:
\begin{itemize}
\itemsep = 0pt
\item $(A_i,B_i)$ corresponds to $(A_{i+a+b},C_{i+a+b})$.
\item $(B_i,C_{i+c})$ corresponds to $(B_{i+2c-2a-2b}, C_{i+3c-2a-2b})$.
\end{itemize}
\end{lemma}

The situation is slightly different in the case where $4a = 0$. Namely, the proof of Proposition~\ref{pro:VTcond} reveals that in this case there are two blue edges that correspond to a given blue edge in the above sense. However, in the case of $4a = 0$ each blue edge clearly lies on four basic $6$-cycle but on each of them the antipodal blue edge is the same. We can thus speak about {\em basic pairs} of blue edges (two blue edges lying on the same basic $6$-cycle). In this sense however each basic pair of blue edges {\em corresponds} in the above way to a unique basic pair of blue edges and it is also clear how to again introduce this concept for blue arcs. Since the ``basic counterpart'' of $(A_i,B_i)$ is $(A_{i+n/2},B_{i+n/2})$ and similarly for the other two types of blue edges, we simplify the notation by writing $\overline{(A_i,B_i)}$ for $\{(A_i,B_i),(A_{i+n/2},B_{i+n/2})\}$ (and similarly for the other types of blue edges). With this agreement it is now easy to see (see also Figure~\ref{fig:alt8_3}) that we get the following. 

\begin{lemma}
\label{le:aux2}
Let $\WH_n(a,b,c,d)$ be a vertex-transitive WH-graph where the parameters $n$, $a$, $b$, $c$ and $d$ are as in Proposition~\ref{pro:VT_basic}. Suppose $4a = 0$ and $4c = 4b$ but $2c \neq 3a+3b$. Then the following holds for each $i \in \ZZ_n$:
\begin{itemize}
\itemsep = 0pt
\item $\overline{(A_i,B_i)}$ corresponds to $\overline{(A_{i+a+b},C_{i+a+b})}$.
\item $\overline{(B_i,C_{i+c})}$ corresponds to $\overline{(B_{i+2c-2b}, C_{i+3c-2b})}$.
\end{itemize}
\end{lemma}

\begin{proposition}
\label{pro:VTcond3}
Let $\G = \WH_n(a,b,c,d)$ be a vertex-transitive WH-graph where the parameters $n$, $a$, $b$, $c$ and $d$ are as in Proposition~\ref{pro:VT_basic}. If $4c = 4a+4b$ but $2c \neq 3a+3b$ then $n = 4m$ for an odd integer $m$, $c = 0$ and, using an isomorphism from Lemma~\ref{le:iso} if necessary, we have that $a = 2$, $b = m-2$ and $d = m+2$. In particular, $\G$ belongs to the family of graphs from Proposition~\ref{pro:VTfam1}.
\end{proposition}

\begin{proof}
As was explained in the paragraphs preceding Lemma~\ref{le:aux1} and Lemma~\ref{le:aux2} we have the notion of corresponding pairs of blue arcs (if $4a \neq 0$) or corresponding pairs of basic pairs of blue arcs. We claim that $a+b$ is of order $4$ (in $\ZZ_n$). We describe how to see this in the case that $4a \neq 0$ but a completely analogous argument works also in the case of $4a = 0$, where one needs to work with correspondence between basic pairs of blue arcs instead of correspondence between blue arcs. Suppose then that $4a \neq 0$ and let us consider the following. We start at a given blue arc $(u,v)$ and then inductively construct a sequence of blue arcs as follows. Let $(u',v')$ be the corresponding blue arc of $(u,v)$. There is then a unique blue arc of the form $(u',w)$, where $w \neq v'$. The blue arc in our sequence that follows $(u,v)$ is then set to be $(u',w)$. Now, using Lemma~\ref{le:aux1} and taking into account the assumption $4c = 4a+4b$ we find that starting with the initial blue arc $(B_0,C_c)$ we get the sequence
$$
	(B_0,C_c) \to (B_{2c-2a-2b},A_{2c-2a-2b}) \to (C_{2c-a-b},B_{c-a-b}) \to (C_{a+b},A_{a+b}) \to (B_0,C_c) \to \cdots
$$
which thus has precisely four distinct elements. As $\cB$ is an $\Aut(\G)$-orbit and we have the automorphism $\tau$ from~\eqref{eq:tau}, the same must hold for the sequence starting with the blue arc $(A_0,B_0)$. Here we get the sequence
$$
	(A_0,B_0) \to (A_{a+b},B_{a+b}) \to (A_{2a+2b}, B_{2a+2b}) \to \cdots,
$$
thus showing that $a+b$ is of order $4$ in $\ZZ_n$, as claimed. 

This implies that $n$ is divisible by $4$ and $4c = 0$. Moreover, no odd prime divisor of $n$ (which of course divides $c$) can divide any of $a$ and $b$ since otherwise it divides both (and thus also $d = 2a+b$), contradicting the assumption that $\G$ is connected. Moreover, since $c$ is even at least one of $a$ and $b$ is odd, and so is coprime to $n$. Now, suppose $b$ is even. Then $a$ is odd and thus coprime to $n$, and so $0 \in b + \la 2a \ra$, contradicting~\eqref{eq:int2}. Therefore, $b$ is odd and thus coprime to $n$. 

If $a$ is odd, then Lemma~\ref{le:iso} implies that we can assume $a = 1$, and so $a+b$ being even and of order $4$ forces $n$ to be divisible by $8$ and $b \in \{n/4-1, 3n/4-1\}$. Using Lemma~\ref{le:iso} with $q = -1$ if necessary we can thus assume that $\{b,d\} = \{n/4 - 1, n/4 + 1\}$. Now, $4c = 0$ implies that $c \in \{0,n/4,n/2,3n/4\}$. Consider the closed walk $(A_0,A_1,A_2,\ldots , A_{n/4-1},C_{n/4-1},B_0,A_0)$ of length $n/4+2$ and note that its first $n/4-1$ edges are red, which are then followed by a blue, a red and another blue edge. Vertex-transitivity and the fact that we have the automorphism $\tau$ imply that there is an automorphism $\eta \in \Aut(\G)$ mapping this walk to a walk starting with the red $2$-path $(B_0,C_{n/4+1},B_2)$. Then $A_{n/4-2}\eta = B_{n/4-2}$ (recall that $n/4$ is even) and thus $A_{n/4-1}\eta = C_{n/2-1}$. Clearly, if $\eta$ maps at least one of the two blue edges of the above walk to a left or a right edge, none of them can be mapped to a $c$-edge. But as $A_0$ and $A_{n/2-1}$ are not adjacent, this shows that $\eta$ maps both blue edges of the walk to $c$-edges, and so $C_{c}$ and $B_{n/2-1-c}$ are adjacent via a red edge, that is $2c+1+n/2 \in \{n/4-1,n/4+1\}$. Since $2c \in \{0,n/2\}$, this is only possible if $n = 8$ and $2c = 4$. But then the blue edge $B_0C_c$ lies on two alternating $4$-cycles, while the blue edge $A_0B_0$ lies on just one, a contradiction. 

This finally shows that $b$ is odd and $a$ is even. Since $4(a+b) = 0$ and $a+b$ is odd, $n \equiv 4 \pmod{8}$. Since the $2$-part of $\gcd(a,n)$ is at least $2$ and the $2$-part of $n$ is $4$, Proposition~\ref{pro:VT_basic} and $4c = 0$ imply that in fact $c = 0$ and the $2$-part of $\gcd(a,n)$ is $2$. In view of Lemma~\ref{le:iso} we can thus assume that $a = 2$. As $a+b$ is of order $4$ in $\ZZ_n$, we must have that $b \in \{n/4-2, 3n/4-2\}$. If $b = n/4-2$ then $d = b+2a = n/4+2$. If however $b = 3n/4-2$, then $d = 3n/4+2$ and using Lemma~\ref{le:iso} with $q = -1$ shows that we can in fact assume $\{b,d\} = \{n/4-2, n/4+2\}$, as claimed.
\end{proof}

This finally proves the main result of this section.

\begin{theorem}
\label{the:mainVT}
Let $\G = \WH_n(a,b,c,d)$ be a WH-graph. Then $\G$ is vertex-transitive if and only if $n$ is even and after possibly changing the roles of $b$, $c$ and $d$, possibly replacing $a$ by $-a$ and possibly using an isomorphism from Lemma~\ref{le:iso} we have that $d = 2a+b$ and one of the following holds:
\begin{enumerate}
\itemsep = 0pt
\item[{\rm (1)}] $\G$ is one of $\WH_8(2,1,0,5)$ and $\WH_8(2,1,4,5)$, or
\item[{\rm (2)}] $a$, $b$ and $d$ are all odd, $c$ is even and $2c = 3a+3b$, or
\item[{\rm (3)}] $n = 4m$ with $m > 1$ odd and $\G = \WH_{4m}(2,m-2,0,m+2)$.
\end{enumerate}
\end{theorem}

We remark that the smallest two members of the family from item (2) of the above theorem are the nonisomorphic graphs $\WH_4(1,3,0,1)$ and $\WH_4(1,3,2,1)$, while the smallest member of the family from item (3) of Theorem~\ref{the:mainVT} is the graph $\WH_{12}(2,1,0,5)$.

%%%%%%%%%%%%%%%
\section{Concluding remarks}
\label{sec:concluding}

As was mentioned in the Introduction, tetravalent vertex-transitive but not arc-transitive graphs can be useful when constructing tetravalent semisymmetric graphs. In particular,~\cite[Theorem~5.2]{PotWil14} states that each worthy (no two vertices have the same set of neighbors) tetravalent semisymmetric graph of girth $4$ arises from what is called a ``suitable LR-structure'' (see~\cite{PotWil14} for the definition). Not to add too much to the length of this paper, let us simply say that the results of the previous sections imply that vertex-transitive WH-graphs are very good candidates for suitable LR-structures. In particular, Lemma~\ref{le:sym} and Theorem~\ref{the:ET} imply that a vertex-transitive WH-graph together with the natural coloring of its edges, described in Section~\ref{sec:VT}, giving rise to the red and blue cycles, is a suitable LR-structure if and only if it has no alternating $4$-cycles and for any of its vertices there exists an automorphism fixing this vertex and two of its neighbors while interchanging the other two neighbors. 

It is easy to see that, of the WH-graphs appearing in Theorem~\ref{the:mainVT}, precisely those from item (2) with the extra condition $a \in \{\pm b, \pm d\}$ have alternating $4$-cycles. Moreover, one can verify that the two sporadic examples $\WH_8(2,1,0,5)$ and $\WH_8(2,1,4,5)$ do admit automorphisms fixing a vertex and two of its neighbors while swapping the remaining two neighbors. Moreover, as the automorphism $\theta$ from the proof of Proposition~\ref{pro:VTfam2} fixes the vertex $C_0$ and each of $B_{-m-2}$ and $B_{-m+2}$ while swapping $B_0$ with $A_0$, the only vertex-transitive WH-graphs that might not give rise to suitable LR-structures, are the ones from item (2) of Theorem~\ref{the:mainVT} for which either $a \in \{\pm b, \pm d\}$ or the only nontrivial automorphism fixing the vertex $A_0$ is the automorphism $\tau$ from Lemma~\ref{le:sym}. 

It is easy to see that if there is a $q \in \ZZ_n$ such that $qa = -a$, $qc = c$ and $qb = d$ (in which case $d = 2a+b$ implies $qd = b$), then the permutation mapping each $A_i$ to $A_{qi}$, $B_i$ to $B_{qi}$ and $C_i$ to $C_{qi}$ is an automorphism of $\WH_n(a,b,c,d)$ fixing $A_0$, $B_0$ and $C_0$ but swapping $A_a$ with $A_{-a}$. Thus all WH-graphs admitting such a $q \in \ZZ_n$ are suitable LR-structures (provided that $a \notin \{\pm b, \pm d\}$). We leave it as an open problem to determine if there are any other LR-structures among the graphs from item (2) of Theorem~\ref{the:mainVT} (that is those that do not admit a ``suitable'' $q$).

\begin{problem}
Determine whether there exist graphs from the second item of Theorem~\ref{the:mainVT} for which $a \notin \{\pm b, \pm d\}$, there is no $q \in \ZZ_n$ with $qa = -a$, $qc = c$ and $qb = d$, but they do admit an automorphism fixing $A_0$, $B_0$ and $C_0$ while swapping $A_a$ and $A_{-a}$. In the case they do, determine all such examples.
\end{problem}

Finally, there is the natural question of which pairs of vertex-transitive WH-graphs are isomorphic. Up to isomorphism, the only three members of the family from item (2) of Theorem~\ref{the:mainVT} of order $24$ are $\WH_8(1,3,2,5)$, $\WH_8(1,7,0,1)$ and $\WH_8(1,7,4,1)$, none of which is isomorphic to any of the two sporadic graphs from item (i) of Theorem~\ref{the:mainVT}. Suppose $\G_1 = \WH_n(a,b,c,d)$ and $\G_2 = \WH_{4m}(2,m-2,0,m+2)$ are isomorphic graphs from items (2) and (3) of Theorem~\ref{the:mainVT}, respectively. Of course, $n = 4m$ has to hold. The blue cycles in $\G_2$ are $3$-cycles, while the red cycles of $\G_1$ are of even length (recall that $a$ is odd). Therefore, a corresponding isomorphism maps the red cycles of $\G_1$ to the red cycles of $\G_2$. But since $m$ is odd and the red cycles of $\G_2$ are of length $2m$, while those of $\G_1$ are of length $4m/\gcd(a,4m)$, which is divisible by $4$, this is not possible. This contradiction shows that the only possible isomorphisms are among the graphs from item (2) of Theorem~\ref{the:mainVT}. Of course, by Lemma~\ref{le:iso} it suffices to consider the examples with $a$ dividing $n$. This brings us to our second open problem.

\begin{problem}
Determine a necessary and sufficient condition for two graphs $\WH_n(a,b,c,d)$ and $\WH_n(a,b',c',d')$ with $a$ dividing $n$, both satisfying the conditions from item (2) of Theorem~\ref{the:mainVT}, to be isomorphic.
\end{problem}

%%%%%%%%%%%%%%%
\section*{Acknowledgments}

The authors are grateful to Primo\v z Poto\v cnik for fruitful conversations on the topic. They would also like to thank the organizers of the 2017 CMO workshop {\em Symmetries of Discrete Structures in Geometry} held in Oaxaca, Mexico, the 2018 and 2022 {\em SIGMAP} conferences held in Morelia, Mexico, and Fairbanks, USA, respectively, and the 2022 {\em Workshop on Symmetries of Graphs}, held in Kranjska Gora, Slovenia, during which a considerable part of the research that lead to the results of this paper was performed.

P. \v Sparl acknowledges financial support by the Slovenian Research Agency (research program P1-0285 and research projects J1-2451 and J1-3001).

%%%%%%%%%%%%%%%

\end{document}